\documentclass[aap, noinfoline]{imsart}
\setattribute{journal}{name}{}
\RequirePackage[OT1]{fontenc}
\usepackage{amsmath}
\usepackage{amsthm}
\usepackage{amsfonts}
\usepackage{amssymb}
\usepackage{graphicx}
\usepackage{enumitem}
\usepackage{physics}
\RequirePackage[colorlinks,citecolor=blue,urlcolor=blue]{hyperref}

\newcommand{\mc}[1]{\mathcal{#1}}

\newcommand{\floor}[1]{ \lfloor #1 \rfloor }

\newcommand{\varep}{ \varepsilon }

\newcommand{\sse} {\subseteq}

\newcommand{\R}{\mathbb{R}}

\newcommand{\E}{\mathbb{E}}

\newcommand{\ra}{\rightarrow}
\newcommand{\toinf}{\ra \infty}

\newcommand{\beq}{\begin{equation}}
\newcommand{\eeq}{\end{equation}}
\newcommand{\mbf}[1]{\mathbf{#1}}
\newtheorem{theorem}{Theorem}[section]
\newtheorem{prop}[theorem]{Proposition}
\newtheorem{lemma}[theorem]{Lemma}
\newtheorem{cor}[theorem]{Corollary}
\theoremstyle{definition}
\newtheorem{definition}{Definition}[section]
\newtheorem*{remark}{Remark}

\newcommand{\p}{\mathbb{P}}
\newcommand{\Var}[1]{\mathrm{Var}(#1)}
\newcommand{\gengraph}{\mbf{G}_{n}}
\newcommand{\geen}{\mbf{G}_n}
\newcommand{\gendeg}{\lambda}
\newcommand{\geegamma}{\mbf{G}_n}
\newcommand{\locallemmabound}{C J^{2/3} p_n^2 \rho_k(n)}
\newcommand{\troot}{\varnothing}
\newcommand{\localprop}{Property GLA}

\newcommand{\btree}{\mbf{T}}
\numberwithin{equation}{section}
\newcommand{\kayn}{\mbf{K}_n}
\newcommand{\remainderbound}{C J p_n^2 (\delta_k^{1/2} + \varep_k(n)^{1/4} + \rho_k(n)^{1/4})}
\newcommand{\jointcouplingprob}{\varep_k(n) + \frac{C (\gendeg_n + 1)^{2k}}{n} + 2 d_{TV}(F_w^{(n)}, F_w)}
\newcommand{\lra}{\longrightarrow}
\newcommand{\beeg}{\mbf{G}}

\begin{document}
\begin{frontmatter}

\title{Central limit theorems for combinatorial optimization problems on sparse Erd\H{o}s-R\'{e}nyi graphs}
\runtitle{Sparse optimization CLTS}

\begin{aug}
    \author{\fnms{Sky} \snm{Cao}\thanksref{t1}\ead[label=e1]{skycao@stanford.edu}}
    \runauthor{Sky Cao}
    \affiliation{Stanford University}
    
    \thankstext{t1}{Research was supported by NSF grant DMS-1501767.}

    \address{Department of Statistics \\ 
            Stanford University \\ 
            Sequoia Hall, 390 Jane Stanford Way \\ 
            Stanford, CA 94305 \\ \\ 
            \printead{e1}}
\end{aug}

\begin{abstract}
For random combinatorial optimization problems, there has been much progress in establishing laws of large numbers and computing limiting constants for the optimal values of various problems. However, there has not been as much success in proving central limit theorems. This paper introduces a method for establishing central limit theorems in the sparse graph setting. It works for problems that display a key property which has been variously called ``endogeny", ``long-range independence", and ``replica symmetry" in the literature. Examples of such problems are maximum weight matching, $\lambda$-diluted minimum matching, and optimal edge cover.
\end{abstract}

\begin{keyword}[class=MSC]
\kwd[Primary ]{60F05}
\kwd[; secondary ]{90C27, 82B44}
\end{keyword}

\begin{keyword}
\kwd{Central limit theorem}
\kwd{combinatorial optimization}
\kwd{Erd\H{o}s-R\'{e}nyi graph}
\kwd{Stein's method}
\kwd{generalized perturbative approach}
\kwd{endogeny}
\kwd{long-range independence}
\kwd{replica symmetry}
\kwd{maximum weight matching}
\kwd{minimum matching}
\kwd{optimal edge cover.}
\end{keyword}

\end{frontmatter}

\tableofcontents

\section{Introduction}

\subsection{Background and motivation}

Combinatorial optimization problems are in essence functions on weighted graphs. By making the underlying weighted graph random, we may obtain a random combinatorial optimization problem. The basic object of study then becomes the optimal value of the problem, which is now a random variable. There are general strategies for establishing laws of large numbers for this random variable in various settings. When the weighted graph comes from Euclidean points (say $n$ i.i.d. points from the unit square), certain subadditive properties may be exploited. For more details and references in the Euclidean setting, see the monographs by Steele \cite{Steele1997} or Yukich \cite{Yuk1998}. 

Unfortunately, in the Euclidean setting, not much is known about limiting constants (a notable exception is the Euclidean bipartite matching problem, due to recent work by Caracciolo et al. \cite{CLPS2014}, and Ambrosio et al. \cite{AST2019}). One of the main difficulties seems to be the correlation between Euclidean distances. We can thus obtain a more tractable mathematical problem by simply making all distances independent. In other words, the random weighted graph is now obtained by starting with a complete graph, and then giving each edge an i.i.d. edge weight. This is called the mean field setting. A related setting is where the graph is a sparse random graph (e.g. Erd\H{o}s-R\'{e}nyi or random regular); this is called the sparse graph setting. In the 1980s, statistical physicists obtained predictions on the limiting constants for various combinatorial optimization problems in the mean field setting. See W\"{a}stlund \cite{Wast2009, Wast2012} for a list of references. In 2001, Aldous \cite{Ald2001} provided the first rigorous proof of one of these predictions, for the minimum matching problem. His general proof strategy was a rigorous version of the Cavity method from statistical physics. It is called the Objective method, or the Local weak convergence method; see Aldous and Steele \cite{AldSte2004} for a survey. The Objective method gives a general purpose approach to computing limiting constants in the mean field and sparse graph settings, and has been applied to other problems; see \cite{AddBerr2013, GNS2005, JasTat2011, Khand2014, KhanSund2014, Salez2013, Wast2009, Wast2012} for an incomplete list. Besides the objective method, there are also other ways of computing limiting constants; \cite{Frieze1985, HessWast2010, LinWast2003, NPS2005, Wast2005, Wast2008, Wast2009b, Wast2010} is an extensive, yet incomplete list.

Thus for many problems, we understand very well the first order behavior, so let us now look at the fluctuations. It is commonly believed that the optimal value for various random combinatorial optimization problems should be asymptotically Normal; see e.g. the discussion in Chatterjee \cite[Section 5]{Ch2014}. However, in the Euclidean setting, we are only aware of two problems for which a central limit theorem has been proven: minimal spanning tree (Alexander \cite{Alex1996}, Kesten and Lee \cite{KestLee1996}), and Euclidean bipartite matching (del Barrio and Loubes \cite{DBL2019}). Minimal spanning tree is particularly amenable to mathematical analysis because there is a greedy algorithm for solving the problem, which leads to many convenient properties. Such properties were also used by Chatterjee and Sen \cite{ChSen2013} to obtain rates of convergence for minimal spanning tree in the Euclidean and lattice settings. The convenient property of Euclidean bipartite matching is that it may be written as an optimal transport problem, and thus techniques from optimal transport theory may be used.

There are also general central limit theorems for functions of Euclidean point processes; see Yukich \cite{Yuk2013} for a survey and references, and see L\`{a}chieze-Rey et al. \cite{LSY2019} for a recent result. However, for combinatorial optimization problems, verifying the conditions of the general theorems seems to be an open problem. 

Turning now to the mean field setting, we are only aware of a central limit theorem for minimal spanning tree (Janson \cite{Janson1995}), whose proof also uses convenient properties of the problem. For minimum matching, a conjecture is given by Hessler and W\"{a}stlund \cite{HessWast2008}. 

In any of the settings mentioned (Euclidean, lattice, mean field, sparse graph), there does not seem to be a general purpose strategy for obtaining central limit theorems. The present paper seeks to make a dent in this direction, for the sparse graph (more specifically, sparse Erd\H{o}s-R\'{e}nyi) setting. In particular, a general central limit theorem is proven, and is applied to give central limit theorems for various combinatorial optimization problems that have been previously studied in the literature. 




\subsection{Setting}

For $\gendeg > 0$, let $p_n := \gendeg /n$, or more generally, $np_n \ra \gendeg$. In words, $\gendeg$ is the asymptotic average vertex degree. With $p_n$ implicit, let $G_n$ be an Erd\H{o}s-R\'{e}nyi graph on $n$ vertices $[n] := \{1, \ldots, n\}$, with edge probability $p_n$. Additionally, $G_n$ will have edge weights, which are i.i.d. from some non-negative distribution $F_w^{(n)}$ which may depend on $n$. However, we will assume that $F_w^{(n)}$ converges in total variation to some distribution $F_w$. With $p_n$ and $F_w^{(n)}$ implicit, we will denote this weighted graph by $\mbf{G}_n$, which may be represented by a pair $(W^n, B^n)$, where $W^n, B^n$ are independent, with $W^n = (w_{ij}^n, 1 \leq i < j \leq n)$, and $B^n = (b_{ij}^n, 1 \leq i  < j \leq n)$. The entries of $W^n$ are i.i.d. from $F_w^{(n)}$, and the entries of $B^n$ are i.i.d. $\mathrm{Bernoulli}(p_n)$. For notational convenience, we will often hide the dependence on $n$ and write $W, B$ instead of $W^n, B^n$. Generic vertices will typically be denoted $v, u$, and as such for edges $e = (v, u)$, we will often write $w_e = w_{(v, u)} = w_{(u, v)}$, and $b_e = b_{(v, u)} = b_{(u, v)}$. Note we still refer to $e = (v, u)$ as an ``edge", even if it is not present in the weighted graph $\geen$.

We will study optimization problems, denoted by a function $f$ which takes weighted graphs as input. Under certain conditions on $f$, we will be able to show that 
\[ \frac{f(\geen) - \E f(\geen)}{\sqrt{\Var{f(\geen)}}} \stackrel{d}{\lra} N(0, 1).\]
In this paper, $C$ will denote a numerical constant which may always taken to be larger, and which may change from line to line, or even within a line.


\section{Main Results}\label{main-results}

In this section, we will introduce the combinatorial optimization problems that are considered, and collect the main results. As an overview, Section \ref{general-result-section} introduces the general central limit theorem (Theorem \ref{main-result}), and Sections \ref{max-weight-section}-\ref{opt-edge-cover} describe applications of the general theorem. The proof of the general theorem (Section \ref{proofs-section}) is by the generalized perturbative approach to Stein's method (introduced by Chatterjee in \cite{Ch2008}, see also his survey \cite{Ch2014}). Although the general theorem is for the sparse graph setting, Section \ref{opt-edge-cover} actually gives a central limit theorem in the mean field setting. The basic idea is that the mean field setting may be approximated by the sparse graph setting; this idea comes from W\"{a}stlund \cite{Wast2009, Wast2012}.


The key assumption in the general theorem is {\localprop}, which is introduced in Definition \ref{local-prop-def}. Various forms of this key assumption have appeared before in the literature: Aldous and Bandyopadhyay \cite{AldBan2005} call it ``endogeny", Gamarnik et al. \cite{GNS2005} call it ``long-range independence", and W\"{a}stlund \cite{Wast2009, Wast2012} calls it ``replica symmetry". The reason for this is because in applying the Objective method to compute the limiting constant for a given combinatorial optimization problem, the key problem-specific step is in verifying endogeny/long-range independence/replica symmetry. Thus this paper may be summarized as follows: in the sparse Erd\H{o}s-R\'{e}nyi setting, if we can compute the limiting constant for a given problem by the Objective method, then (assuming certain reasonable technical conditions) we also get a central limit theorem.



The first two combinatorial optimization problems we introduce deal with matchings on graphs. A matching of a graph is a collection of edges such that each vertex is incident to at most one edge of the collection. Given a weighted graph, we may naturally define the weight of a matching to be the total sum of edge weights over edges in the matching. 

\subsection{Maximum weight matching}\label{max-weight-section}

Given a weighted graph, we define the maximum weight matching to be the matching with maximal weight. Fix $\lambda > 0$. Let $\geen$ be the weighted graph with $p_n = \lambda / n$, and i.i.d. $\mathrm{Exp}(1)$ edge weights (so in this case, the edge weight distribution does not depend on $n$). Let $M_n = M(\geen)$ be the weight of the maximum weight matching of $\geen$. In \cite{GNS2005}, it was proven that maximum weight matching possesses long-range independence, which allowed the authors to show $M_n / n \stackrel{p}{\ra} \beta(\lambda)$, with $\beta(\lambda)$ explicitly characterized. The following theorem gives a central limit theorem for $M_n$.


\begin{theorem}\label{max-weight-match-normal}
We have
\[ \frac{M_n - \E M_n}{\sqrt{\Var{M_n}}} \stackrel{d}{\lra} N(0, 1). \]
\end{theorem}

\subsection{\texorpdfstring{$\lambda$-diluted minimum matching}{}}\label{lambda-min-match}

Fix $\lambda > 0$. Let $\kayn$ denote the complete graph with i.i.d. edge weights distributed as $n \text{Exp}(1)$. For a matching of $\kayn$, define its {\it $\lambda$-diluted cost} as the sum of the edge weights in the matching, plus $\lambda / 2$ times the number of unmatched vertices. Let $M_\lambda(\kayn)$ be the minimal $\lambda$-diluted cost among all matchings. 

A priori, this problem doesn't seem to be an optimization problem on a sparse graph. But observe that in finding the $\lambda$-diluted minimum matching, we may ignore all edges in $\kayn$ of weight larger than $\lambda$. Thus $M_\lambda$ is actually a function of $\kayn(\lambda)$, the subgraph of $\kayn$ consisting of all edges with weight at most $\lambda$. Observe that $\kayn(\lambda)$ is exactly a weighted sparse Erd\H{o}s-R\'{e}nyi graph where  $p_n = 1 - e^{-\lambda / n} \approx \lambda / n$, and $F_w^{(n)}$ is the distribution of $n \text{Exp}(1)$ conditioned to lie in $[0, \lambda]$ (which converges in total variation to $\text{Unif}[0, \lambda]$ as $n \toinf$). Thus $\lambda$-diluted minimum matching can be made to fit into our framework, and thus we have a central limit theorem. 

\begin{theorem}\label{lambda-diluted-match-normal}
We have
\[ \frac{M_\lambda(\kayn) - \E M_\lambda(\kayn)}{\sqrt{\Var{M_\lambda(\kayn)}}} \stackrel{d}{\lra} N(0, 1). \]
\end{theorem}

To give some background, this problem was introduced by W\"{a}stlund \cite{Wast2009, Wast2012}, see also his paper with Parisi \cite{ParWast2017}. These papers all study the minimum matching problem, which is the $\lambda \toinf$ limit of the $\lambda$-diluted minimum matching problem. The $\lambda$-diluted problem was introduced as a more localized version of minimum matching, and as such proved easier to analyze. Replica symmetry was shown for $M_\lambda$, which led to the proof of $M_\lambda(\kayn) / n \stackrel{p}{\ra} \beta(\lambda)$, with $\beta(\lambda)$ explicitly characterized. Results on minimum matching were then deduced from the results on the $\lambda$-diluted problem by taking $\lambda \toinf$. In this way one can think of the minimum matching problem on a complete graph as ``essentially" a sparse graph problem.


\subsection{Optimal edge cover}\label{opt-edge-cover}

As before, let $\kayn$ be the complete graph with i.i.d. edge weights distributed as $n \text{Exp}(1)$. An edge cover is a collection of edges such that each vertex of $\kayn$ is incident to at least one edge of the collection. Naturally, the cost of an edge cover is the sum of edge weights over all edges in the edge cover. The optimal edge cover is defined to be the edge cover of minimal weight, and its cost is denoted $EC(\kayn)$. 

\begin{theorem}\label{opt-edge-cover-clt}
We have
\[ \frac{EC(\kayn) - \E EC(\kayn)}{\sqrt{\Var{EC(\kayn)}}} \stackrel{d}{\lra} N(0, 1). \]
\end{theorem}

The proof of this will be through what is essentially a truncation argument. Similar to the previous section, we will define a certain relaxed version of optimal edge cover, indexed by a parameter $\lambda > 0$, and denoted $EC_\lambda(\kayn)$. The relaxed problem is related to optimal edge cover by
\[ \lim_{\lambda \toinf} EC_\lambda(\kayn) = EC(\kayn). \]
Moreover, the relaxed problem is an optimization problem on a sparse Erd\H{o}s-R\'{e}nyi graph, so that the methods of this paper will apply to give a central limit theorem for $EC_\lambda(\kayn)$. Even more, we will have a rate of convergence, which with a little bit of work, can be shown to be robust enough to allow us to take $\lambda$ to infinity with $n$. This will then allow us to transfer the central limit theorem for $EC_\lambda(\kayn)$ to a central limit theorem for $EC(\kayn)$.

The relaxed problem was introduced by W\"{a}stlund \cite{Wast2009}. First take $\lambda > 0$. We may define the $\lambda$-diluted cost of a collection of edges (not necessarily an edge cover) as the sum of edge weights over all edges in the collection, plus $\lambda / 2$ times the number of vertices that are not incident to any edge of the collection. In effect, we are paying a penalty of $\lambda / 2$ for each un-covered vertex. The optimal $\lambda$-diluted edge cover is defined to be the collection of edges (again, not necessarily an edge cover) of minimal $\lambda$-diluted cost, and its cost is denoted $EC_\lambda(\kayn)$. Observe $EC_\lambda$ is actually just a function of $\kayn(\lambda)$, and thus is an optimization problem on a sparse Erd\H{o}s-R\'{e}nyi graph. W\"{a}stlund \cite{Wast2009} showed replica symmetry for $EC_\lambda$, which led to the proof of $EC_\lambda(\kayn) \stackrel{p}{\ra} \beta(\lambda)$, with $\beta(\lambda)$ explicitly characterized.

\subsection{The general result}\label{general-result-section}

We start by giving the definitions needed to state the general result. 

\begin{definition}
Given a weighted graph $\mbf{G}$, and a vertex $v$, let $\mbf{G} - v$ denote the weighted graph obtained by deleting $v$ and all edges incident to $v$. More generally, given a set of vertices $U$, let $\mbf{G} - U$ denote the weighted graph obtained by deleting all vertices in $U$, as well as all edges which are incident to a vertex of $U$.
\end{definition}

\begin{definition}
Given an integer $k \geq 0$, let $B_k(v, \mbf{G})$ denote the weighted graph obtained from the union of all paths in $\mbf{G}$ which start at $v$ and are of length at most $k$. We will call $v$ the root of $B_k(v, \mbf{G})$, even if $B_k(v, \mbf{G})$ is not a tree. We may think of $B_k(v, \mbf{G})$ as a (small) neighborhood of $v$. 
\end{definition}

We may think of $B_k(v, \mbf{G})$ as a (small) neighborhood of $v$. One key technical fact is that we can replace $B_k(v, \gengraph)$ by a limiting object, which we now define.

\begin{definition}
For $\gendeg > 0$, let $T(\infty, \gendeg)$ denote a Galton-Watson process with offspring distribution $\mathrm{Poisson}(\gendeg)$. For integer $k > 0$, let $T(k, \gendeg)$ denote the depth $k$ subtree of $T(\infty, \gendeg)$. Additionally, given a weight distribution $F_w$, let $\mbf{T}(\infty, \gendeg, F_w)$ denote $T(\infty, \gendeg)$, equipped with edge weights which are i.i.d. from $F_w$. Let $\mbf{T}(k, \gendeg, F_w)$ denote the weighted depth $k$ subtree of $\btree(\infty, \gendeg, F_w)$. We will typically denote the root of these trees by $\troot$.
\end{definition}

The point is that $B_k(v, \gengraph)$ is essentially $\btree_k$, in the following sense (see Section \ref{er-graph-facts} for precise statements).

\begin{definition}\label{nbd-iso}
Let $\beeg$ be a weighted graph, and let $v$ be a vertex of $\beeg$. Let $\btree$ be a rooted weighted tree. Let $k > 0$. We say $B_k(v, \beeg) \cong \btree$ if there exists a bijection $\varphi$ between the vertices of $B_k(v, \beeg)$ and the vertices of $\btree$, which maps $v$ to the root of $\btree$, and preserves all edges and edge weights. In other words, if $(u, u')$ is an edge of $B_k(v, \beeg)$ with weight $w$, then $(\varphi(u), \varphi(u'))$ is an edge of $\btree$ with weight $w$, and vice versa.
\end{definition}

In words, $B_k(v, \mbf{G}) \cong \btree$ if the two objects differ only by a vertex relabeling. We extend Definition \ref{nbd-iso} to pairs of neighborhoods which share the same root.

\begin{definition}\label{pair-tree-iso}
Let $\beeg, \beeg'$ be weighted graphs which share the same vertex set. Let $v$ be a vertex of $\beeg, \beeg'$. Let $\btree, \btree'$ be trees which share the same root. Let $k > 0$. We say that $(B_k(v, \beeg), B_k(v, \beeg')) \cong (\btree, \btree')$, if there exists a bijection $\varphi$ between the vertices of $B_k(v, \beeg), B_k(v, \beeg')$ and the vertices of $\btree, \btree'$, which maps $v$ to the root of $\btree, \btree'$, and preserves all edges and edge weights. In other words, if $(u, u')$ is an edge of $B_k(v, \beeg)$ with weight $w$, then $(\varphi(u), \varphi(u'))$ is an edge of $\btree$ with weight $w$, and vice versa. Similarly, if $(u, u')$ is an edge of $B_k(v, \beeg')$ with weight $w$, then $(\varphi(u), \varphi(u'))$ is an edge of $\btree'$ with weight $w$, and vice versa.
\end{definition}

For technical reasons, we will also need to work with the following objects.

\begin{definition}
Given $\gendeg > 0$, and a weight distribution $F_w$, define $\tilde{\btree}(\infty, \gendeg, F_w)$ as follows. Take $\btree, \btree' \stackrel{i.i.d.}{\sim} \btree(\infty, \gendeg, F_w)$. Let $\troot, \troot'$ denote the roots of $\btree, \btree'$ respectively. Construct $\tilde{\btree}(\infty, \gendeg, F_w)$ as the tree with root $\troot$, obtained by starting with $\btree$, and then adding an edge between $\troot, \troot'$ with edge weight distributed as $F_w$, independent of everything else. Let $\tilde{\btree}(k, \gendeg, F_w)$ be the depth $k$ subtree of $\tilde{\btree}(\infty, \gendeg, F_w)$.

Note that $\tilde{\btree}_k \stackrel{d}{=} \tilde{\btree}(k, \gendeg, F_w)$ may be constructed in the following manner. Take $\btree_k, \btree_{k-1}'$ independent, with $\btree_k \stackrel{d}{=} \btree(k, \gendeg, F_w)$, $\btree_{k-1}' \stackrel{d}{=} \btree(k-1, \gendeg, F_w)$, with roots $\troot, \troot'$ respectively. Let $\ell \sim F_w$ independent of everything else. Then define $\tilde{\btree}_k$ to be the tree with root $\troot$ constructed by connecting $\troot, \troot'$ with an edge of weight $\ell$. When $\tilde{\btree}_k$ is defined this way, we say that $\tilde{\btree}_k$ is constructed from $(\btree_k, \btree_{k-1}', \troot, \troot', \ell)$.
\end{definition}

\begin{remark}
Observe that the underlying graph of $\tilde{\btree}(\infty, \gendeg, F_w)$ is a Galton-Watson process, where the root has offspring distribution $1 + \mathrm{Poisson}(\gendeg)$, and every subsequent vertex has offspring distribution $\mathrm{Poisson}(\gendeg)$.
\end{remark}






\begin{definition}
Let $(W', B')$ be an i.i.d. copy of $(W, B)$. Given an edge $e$, define $\gengraph^e$ to be the weighted graph obtained by using $w_e', b_e'$ in place of $w_e, b_e$. Let $\Delta_e f := f(\geen) - f(\geen^e)$.
\end{definition}

We now present the key assumption that an optimization problem must satisfy for us to be able to prove a central limit theorem. Roughly speaking, it says that small perturbations of the problem must be able to be locally approximated. See Sections 3 or 5 of \cite{Ch2014} for the motivation for making such an assumption.

\begin{definition}\label{local-prop-def}
Let $f$ be a function on weighted graphs. We say that $(f, (\gengraph, n \geq 1))$ has {\bf \localprop} (``good local approximation") for $\gendeg, F_w$, if $np_n \ra \gendeg$, $d_{TV}(F_w^{(n)}, F_w) \ra 0$, and for each $k > 0$ there exist functions $LA^L_k, LA^U_k$, which take as input pairs of finite rooted weighted trees, such that the following conditions hold.
\begin{enumerate}
    \item[(A1)] For any edge $e = (v, u)$, if $B_k := B_k(v, \geen)$ and $B_k' := B_k(v, \geen^e)$ are trees, then
    \[ LA^L_k(B_k, B_k') \leq \Delta_e f \leq LA^U_k(B_k, B_k'). \]
    \item[(A2)] 
    For any edge $e = (v, u)$, if $(\btree, \btree')$ is such that we have 
    \[ (B_k, B_k') := (B_k(v, \geen), B_k(v, \geen^e)) \cong (\btree, \btree'), \]
    then
    \[ LA_k^L(\btree, \btree') = LA_k^L(B_k, B_k'), \] 
    \[ LA_k^U(\btree, \btree') = LA_k^U(B_k, B_k').\]
    \item[(A3)] 
    Let $\tilde{\btree}_k \stackrel{d}{=} \tilde{\btree}(k, F_w, \gendeg)$ be constructed from $(\btree_k, \btree_{k-1}', \troot, \troot', \ell)$. 
    Define
    \[ \begin{split}\delta_k := \max \bigg(&\E (LA^U_k(\tilde{\btree}_k, \btree_k) - LA^L_k(\tilde{\btree}_k, \btree_k))^2, \\ & \E (LA^U_k(\btree_k, \tilde{\btree}_k) - LA^L_k(\btree_k, \tilde{\btree}_k))^2 \bigg). \end{split}\]
    Then
    \[ \lim_{k \toinf} \delta_k = 0. \]
\end{enumerate}
\end{definition}


\begin{remark}
In words, {\localprop} ensures that when we perturb the weighted graph $\geen$ at a single edge $e$, the resulting change in $f$ may be approximated by $LA^L_k$. Then by (A1), the approximation error is at most $LA^U_k - LA^L_k$. In analyzing this term, (A2) allows us to replace neighborhoods of Erd\H{o}s-R\'{e}nyi graphs by Galton-Watson trees. Then by (A3), we may conclude that the approximation error goes to 0 as $k$ goes to infinity.

The construction of $LA^L_k, LA^U_k$ to fulfill (A1) usually proceeds by exploiting certain recursive properties of the given function $f$. Typically, (A2) will be trivial to check, because in constructing the local approximations, we will never use the vertex labels. The recursive distributional properties of Galton-Watson trees are usually used to verify (A3).
\end{remark}

\begin{remark}
As we will see, the previously mentioned properties of long-range independence and replica symmetry imply {\localprop} in the examples we consider. For instance, in the case of $\lambda$-diluted minimum matching, construction of $LA^L_k, LA^U_k$ was more or less done by W\"{a}stlund, except they were called $f_A^k, f_B^k$ (see \cite[Section 2.5]{Wast2012}). The verification of (A3) in this case is essentially \cite[Proposition 2.8]{Wast2012}. The same story should typically be true for combinatorial optimization problems on sparse Erd\H{o}s-R\'{e}nyi graphs for which the Objective method can be used to derive the limiting behavior of the mean.

As for the relation between {\localprop} and endogeny, first, as mentioned in the previous remark, the local approximations $LA^L_k, LA^U_k$ are usually constructed by a problem specific recursion. If this is the case, then this recursion can also be used to define a recursive tree process. The condition in (A3) that $\lim_{k \toinf} \delta_k = 0$ then essentially implies that the recursive tree process is endogenous (see \cite[Section 2]{AldBan2005} for definitions of recursive tree process and endogeny). Thus {\localprop} should typically imply endogeny (of an associated recursive tree process).

We do hesitate to claim that the properties mentioned here are always equivalent to each other, but in each of the examples considered in this paper, the various properties all follow by a single, central argument. Thus needless to say, these properties are all interrelated. We chose to assume {\localprop} in this paper (rather than any of the other properties) because the assumptions are stated in a way that is suitable for proving a central limit theorem. These assumptions seek to abstract out precisely what is needed to apply the generalized perturbative approach to Stein's method -- no more, no less.
\end{remark}

We now state the general result, which says that any optimization problem which has {\localprop}, with some additional regularity conditions, satisfies a central limit theorem.

\begin{theorem}\label{main-result}
Suppose $(f, (\geen, n \geq 1))$ satisfies {\localprop} for $\gendeg, F_w$. Suppose additionally that the following regularity condition is satisfied.
There is a function $H$ such that 
\beq\label{H-moment} J := \max\bigg(1, \sup_n \E H(w_e, w_e')^6\bigg) < \infty, \eeq
(here the the dependence on $n$ comes from $w_e, w_e'$, which are distributed like $F_w^{(n)}$), and for any $e = (v, u)$ and any $n$, we have
\beq\label{f-regularity} \abs{\Delta_e f} \leq 1(\max(b_e, b_e') = 1) H(w_e, w_e'). \eeq
Let $\sigma_n^2 := \Var{f(\gengraph)}$, $\gendeg_n := np_n$, 
\[ Z_n := \frac{f(\gengraph) - \E f(\gengraph)}{\sigma_n}, \]
and let $\Phi$ denote the standard normal cdf. There is a numerical constant $C_0$, such that with
\[\varep_k(n) := \frac{(2\gendeg + 3)^{k}}{n^{1/3}} + \frac{C_0 (\gendeg_n + 1)^{k}}{\min(\gendeg, 1)} \bigg(\abs{\gendeg_n - \gendeg} + d_{TV}(F_w^{(n)}, F_w) + \frac{\lambda^2}{2n}\bigg), \]
and 
\[ \rho_k(n) :=  \min\bigg(\frac{(\gendeg_n + C_0)^{2k + C_0}}{n}, 1\bigg), \]
we have for any $k, n > 0$,
\[ \begin{split}
&\sup_{t \in \R} \abs{\p(Z_n \leq t) - \Phi(t)} \leq \\
 &C_0 J^{1/4} \Bigg[\bigg(\frac{n}{\sigma_n^2}\bigg)^{1/2} \bigg( \delta_k^{1/8} + \varep_k(n)^{1/16} +
\rho_k(n)^{1/16} \bigg) +\bigg(\frac{n}{\sigma_n^2}\bigg)^{3/4} \frac{\gendeg_n^{1/2}}{n^{1/4}}\Bigg].
\end{split}\]
\end{theorem}

\begin{remark}
To help parse the rate of convergence, note that by assumption, $\delta_k \ra 0$, $\lambda_n \ra \lambda$, and for all $k$, $\lim_{n \toinf} \varep_k(n) = 0$, $\lim_{n \toinf} \rho_k(n) = 0$. Thus to obtain convergence, one naturally will first take $n\toinf$, and then $k \toinf$. This will be successful as long as one is able to show that the variance $\sigma_n^2$ is at least of order $n$. This variance lower bound may in general be nontrivial to obtain. However, for various problems, one may use the general method introduced by Chatterjee \cite{Ch2018}. The other regularity conditions should be easier to verify.
\end{remark}

The following corollary gives sufficient conditions under which {\localprop} holds. It also simplifies the conclusion of the previous theorem, at the cost of no longer giving a rate of convergence.

\begin{cor}\label{cor-main-result}
Suppose we have $(f, (\geen, n \geq 1))$, such that $np_n \ra \gendeg$, $d_{TV}(F_w^{(n)}, F_w) \ra 0$. Suppose for each $k$ there exists functions $g^L_k, g^U_k$ on finite rooted weighted trees, such that when $B_k(v, \geen)$ is a tree, we have
\beq\label{add-vertex-bracket} g^L_k(B_k(v, \geen)) \leq f(\geen) - f(\geen - v) \leq g^U_k(B_k(v, \geen)). \eeq
Suppose moreover that if $B_k(v, \geen) \cong \btree$, then 
\beq\label{gk-label-invariant} g^{L}_k(\btree) = g^L_k(B_k(v, \geen)), ~g^U_k(\btree) = g^U_k(B_k(v, \geen)). \eeq
Additionally, let $\btree_k \stackrel{d}{=} \btree(k, \gendeg, F_w)$, $\tilde{\btree}_k \stackrel{d}{=} \tilde{\btree}(k, \gendeg, F_w)$, and suppose
\beq\label{gu-gl-conv-0} \lim_{k \toinf} \E (g^U_k(\btree_k) - g^L_k(\btree_k))^2 = 0, \eeq
\beq\label{gu-gl-exp-tree-conv-0} \lim_{k \toinf} \E (g^U_k(\tilde{\btree}_k) - g^L_k(\tilde{\btree}_k))^2 = 0. \eeq
Suppose also that the following regularity conditions are satisfied.
\begin{itemize}
    \item The variance is at least of order $n$:
    \[ \liminf_{n \toinf} n^{-1} \mathrm{Var}(f(\geen)) > 0. \]
    \item There exists a function $H$ such that
    \[ \sup_n \E H(w_e, w_e')^6 < \infty, \]
    and for any $e = (v, u)$ and any $n$, we have
    \[\abs{\Delta_e f} \leq 1(\max(b_e, b_e') = 1) H(w_e, w_e'). \]
\end{itemize}
Then
\[ \frac{f(\geen) - \E f(\geen)}{\sqrt{\Var{f(\geen)}}} \stackrel{d}{\lra} N(0, 1).\]
\end{cor}

\begin{remark}
Theorem \ref{main-result} and Corollary \ref{cor-main-result} seek to hide away as many technical details involving Erd\H{o}s-R\'{e}nyi graphs as possible. So to prove a central limit theorem, one may work almost exclusively with Galton-Watson trees, which due to their recursive nature, are much nicer objects.
\end{remark}

\begin{remark}
To point out some relations with previous work, {\localprop} is somewhat reminiscent of the stabilization assumption (see e.g. \cite[Section 8.3]{Yuk2013}) that appears for functions of Euclidean point processes, in that both assumptions control the change in the function when the underlying graph is slightly perturbed. It should also be mentioned that the overarching idea of the proof of Theorem \ref{main-result} essentially already appears in \cite[Section 4]{Ch2014} (see also \cite{ChSen2013}). In another direction, {\localprop} is related to the idea that sums of local statistics of random graphs should be asymptotically Normal. This idea is not new, and has appeared for instance in the context of Euclidean point processes \cite[Section 3.4]{Ch2008}, as well as in the context of configuration models \cite{AY2018, BR2019}. 
\end{remark}

\section{Applications of Corollary \ref{cor-main-result}}\label{general-theorem-applications}

In this section, we will apply the simpler Corollary \ref{cor-main-result} to the first two combinatorial optimization problems listed in Section \ref{main-results}. In both cases, assumption \eqref{gk-label-invariant} will be clear from construction of the $g^L_k, g^U_k$.

For rooted weighted trees $\btree$, we will denote the root by $\varnothing$. For vertices $u \in \btree$, we will denote the set of children of $u$ by $\mc{C}(u)$. Additionally, for edges $(v, u)$ in $\mbf{T}$ with $u$ the child, we will denote the edge weight by $\ell_u$.

\subsection{Maximum weight matching}

In this problem, we have that $p_n = \lambda / n$, and the weight distribution is $\text{Exp}(1)$ for all $n$. The ideas of \cite{GNS2005} will allow us to verify {\localprop}. we detail them here with no claims of originality. 

\begin{proof}[Construction of $g^L_k, g^U_k$]
For $v \in \mbf{G}$, observe that we have the recursion
\[ M(\mbf{G}) = \max\bigg(M(\mbf{G} - v), \max_{u : (v, u) \in  \mbf{G}} w_{(v, u)}  + M(\mbf{G} - \{v, u\}) \bigg).\]
Defining $h(\mbf{G}, v) := M(\mbf{G}) - M(\mbf{G} - v)$, we thus have
\beq\label{max-weight-match-recurs} h(\mbf{G}, v) = \max\bigg(0, \max_{u : (v, u) \in \mbf{G}} w_{(v, u)} - h(\mbf{G} - v, u)\bigg). \eeq
We will use this recursion to define the local approximations. Given an integer $k > 0$, and a rooted weighted tree $\mbf{T}$ of depth at most $k$, define $h_k(\cdot~; \btree) : \btree \ra \R$ in the following manner. For all leaf vertices $u \in \btree$, set $h_k(u; \btree) := 0$. Then use \eqref{max-weight-match-recurs} to define $h_k$ at all other vertices. In other words, for non-leaf vertices $u \in \mbf{T}$, set
\[ h_k(u; \btree) := \max\bigg(0, \max_{u' \in \mc{C}(u)} \ell_{u'} - h_k(u'; \btree)\bigg).\]
One may verify by using induction that for any even $k$ such that $B_k(v, \geegamma)$ is a tree, we have
\[ h_k(v; B_k(v, \geegamma)) \leq h(\geegamma, v), \]
and for any odd $k$ such that $B_k(v, \geegamma)$ is a tree, we have
\[ h(\geegamma, v) \leq h_k(v; B_k(v, \geegamma)).\]
Thus for any odd $k$ such that $B_k(v ,\geegamma)$ is a tree, we have
\[ h_{k-1}(v; B_{k-1}(v, \geegamma)) \leq h(\geegamma, v) \leq h_k(v; B_k(v, \geegamma)). \]
Now to define $g^L_k, g^U_k$, let $i_L := 2 \floor{(k-1)/2}$, and $i_U := 2 \floor{(k-1)/2} + 1$.  In other words, $i_U$ is the largest odd number less than or equal to $k$, and $i_L = i_U - 1$.  This definition ensures that $i_L, i_U \leq k$, and thus we may set (when $B_k(v, \geegamma)$ is a tree) 
\[ g^L_k(B_k(v, \geegamma)) := h_{i_L}(v; B_{i_L}(v, \geegamma)), ~g^U_k(B_k(v, \geegamma)) := h_{i_U} (v; B_{i_U}(v, \geegamma)). \] 
With this definition, \eqref{add-vertex-bracket} is satisfied.
\end{proof}

\begin{proof}[Verification of \eqref{gu-gl-conv-0}, \eqref{gu-gl-exp-tree-conv-0}]
Let $\mbf{T}_\infty \stackrel{d}{=} \mbf{T}(\infty, \lambda, \text{Exp}(1))$, and let $\mbf{T}_k$ be the depth $k$ subtree of $\mbf{T}_\infty$. For brevity, we write $h_k(\troot)$ instead of $h_k(\troot; \btree_k)$. 
To verify \eqref{gu-gl-conv-0}, it suffices to show
\beq\label{max-match-hb-ha} \lim_{r \toinf} \E (h_{2r+1}(\troot) - h_{2r}(\troot))^2 = 0.\eeq
With this coupling of the trees $(\btree_k, k \geq 1)$, one may verify that 
\beq\label{max-match-h-k-monotone} h_{2r+1}(\troot) \text{ is non-increasing in $r$ }, ~~ h_{2r}(\troot) \text{ is non-decreasing in $r$. } \eeq
Observe also that for all $r$,
\[ h_{2r}(\troot) \leq h_{2r+1}(\troot). \]
Defining $h^U := \lim_{r \toinf} h_{2r+1}(\troot)$, $h^L:= \lim_{r \toinf} h_{2r}(\troot)$, we thus have that
\[ h_{2r+1}(\troot) - h_{2r}(\troot) \downarrow h^U - h^L. \]
Note
\[0 \leq h_{2r+1}(\troot) - h_{2r}(\troot) \leq h_{2r+1}(\troot) \leq h_1(\troot) \leq \max_{u \in \mc{C}(\troot)} \ell_u, \]
and the quantity on the right hand side has finite second moment. Thus by dominated convergence, to verify \eqref{max-match-hb-ha}, it suffices to show that $h^U - h^L = 0$ a.s. As $h^L \leq h^U$, the following lemma suffices.

\begin{lemma}
$\E h^L = \E h^U$.
\end{lemma}
\begin{proof}
It follows by Theorem 3 and Proposition 1 of \cite{GNS2005} that $h_k(\varnothing) \stackrel{d}{\lra} X_*$, for some $X_*$. This implies $h^L \stackrel{d}{=} h^U$, and thus $\E h^L = \E h^U$.
\end{proof}
\begin{remark}
In a sense, everything before this lemma is routine, while the assertion that $h^L \stackrel{d}{=} h^U$ is nontrivial. This is one of the major results of \cite{GNS2005}, and it is essentially this assertion that is refered to as ``long-range independence" by Gamarnik et al.
\end{remark}


Once we've verified \eqref{gu-gl-conv-0}, \eqref{gu-gl-exp-tree-conv-0} follows easily. Let $\tilde{\btree}_k$ be constructed from $(\btree_k, \btree_{k-1}', \troot, \troot', \ell)$. Moreover, we may assume that $\tilde{\btree}_k, \tilde{\btree}_{k+1}$ are coupled so that $\tilde{\btree}_k$ is the depth $k$ subtree of $\tilde{\btree}_{k+1}$. It then suffices to show
\[ \lim_{r \toinf} \E (h_{2r+1}(\troot; \tilde{\btree}_{2r+1}) - h_{2r}(\troot; \tilde{\btree}_{2r}))^2 = 0. \]
Observe
\[ h_{2r+1}(\troot; \tilde{\btree}_{2r+1}) = \max\bigg(h_{2r+1}(\troot; \btree_{2r+1}), \ell - h_{2r}(\troot'; \btree'_{2r})\bigg), \]
and
\[ h_{2r}(\troot; \tilde{\btree}_{2r}) = \max\bigg(h_{2r}(\troot; \btree_{2r}), \ell - h_{2r-1}(\troot'; \btree'_{2r-1}) \bigg).\]
Letting $X_r := h_{2r+1}(\troot; \btree_{2r+1}) - h_{2r}(\troot; \btree_{2r})$, $X_r' := h_{2r-1}(\troot'; \btree'_{2r-1}) - h_{2r}(\troot'; \btree'_{2r})$, we have
\[ 0 \leq h_{2r+1}(\troot; \tilde{\btree}_{2r+1}) - h_{2r}(\troot; \tilde{\btree}_{2r}) \leq X_r + X_r', \]
and thus
\[ \bigg(h_{2r+1}(\troot; \tilde{\btree}_{2r+1}) - h_{2r}(\troot; \tilde{\btree}_{2r})\bigg)^2 \leq 2 X_r^2 + 2 (X_r')^2. \]
We've already shown $\E X_r^2 \ra 0$, and a small modification of the proof also shows $\E (X_r')^2 \ra 0$, and thus \eqref{gu-gl-exp-tree-conv-0} is verified.
\end{proof}

With {\localprop} established, we proceed to verify the regularity conditions of Corollary \ref{cor-main-result}. Fix $e = (v, u)$. To determine the function $H$, by splitting into the cases $b_e = 0, 1$, $b_e' = 0, 1$, we may obtain
\[ \abs{M(\geen) - M(\geen^e)} \leq 1(\max(b_e, b_e') = 1) \max(w_e, w_e'). \]
Thus we may take $H(w_e, w_e') := \max(w_e, w_e')$. As $w_e, w_e' \stackrel{i.i.d.}{\sim} \mathrm{Exp}(1)$, clearly $\E H(w_e, w_e')^6 < \infty$.


The application of Corollary \eqref{cor-main-result} to prove Theorem \eqref{max-weight-match-normal} will now be complete as soon as we show the following variance lower bound.

\begin{lemma}\label{max-match-var-lower-bd}
We have
\[ \liminf_{n \toinf} n^{-1} \Var{M(\geen)} > 0.\]
\end{lemma}
\begin{proof}
We use the general framework of \cite{Ch2018}. 
For brevity, let $M_n := M(\geen)$. As observed in \cite{Ch2018}, it suffices to find constants $c_1, c_2 > 0$ such that for large enough $n$, for $b - a \leq c_1 \sqrt{n}$, we have
\[ \p(a \leq M_n \leq b) \leq 1 - c_2. \]

To find $c_1, c_2$, first observe that conditional on the underlying graph $G_n$, the law of $\geen$ is some structured collection of i.i.d. $\mathrm{Exp}(1)$ random variables, call them $w_1, \ldots, w_{E_n}$, where $E_n$ is the number of edges in $G_n$. For $\alpha > 0$ to be chosen later, set $\varep := \varep_n := \alpha n^{-1/2}$, and $w_i' := w_i / (1 - \varep)$, $1 \leq i \leq E_n$. Let $M_n'$ be the maximum weight matching of $G_n$ with the edge weights $w_1', \ldots, w_{E_n}'$. Lemma 1.2 of \cite{Ch2018} implies that for $-\infty < a \leq b < \infty$, we have
\beq\label{prob-a-b-bd} \p(a \leq M_n \leq b) \leq \frac{1}{2}(1 + \p(\abs{M_n - M_n'} \leq b - a) + d_{TV}(\mc{L}_{M_n}, \mc{L}_{M_n'})), \eeq
where $d_{TV}(\cdot, \cdot)$ is total variation distance, and $\mc{L}_{M_n}, \mc{L}_{M_n'}$ are the laws of $M_n$, $M_n'$, respectively. Let $d_{TV}(\cdot, \cdot ~|~ G_n)$ denote total variation distance conditional on $G_n$. Then it follows by Corollary 1.8 of \cite{Ch2018} that
\[ d_{TV}(\mc{L}_{M_n}, \mc{L}_{M_n'} ~|~ G_n) \leq C(E_n \alpha^2 / n)^{1/2} = C (E_n / n)^{1/2}\alpha.\]
Thus
\[ d_{TV}(\mc{L}_{M_n}, \mc{L}_{M_n'}) \leq \E d_{TV}(\mc{L}_{M_n}, \mc{L}_{M_n'} ~|~ G_n) \leq C \alpha \E (E_n / n)^{1/2} \leq C \sqrt{\lambda} \alpha,  \]
where the final inequality follows by noting $E_n \sim \text{Binomial}(n(n-1) / 2, \lambda / n)$. 

Observe now that $M_n' = M_n / (1-\varep)$, and thus $\abs{M_n - M_n'} = M_n \varep / (1 -\varep)$. By Theorem 3 of \cite{GNS2005}, we have $M_n / n \stackrel{p}{\ra} \beta(\lambda) > 0$. In particular, for some $c_1 > 0$ small enough, we have
\[ \p(M_n \varep / (1 - \varep) \leq c_1 \sqrt{n}) \leq \p(M_n / n \leq c_1 / \alpha) \ra 0.\]
We now choose $\alpha$ small so that $d_{TV}(\mc{L}_{M_n}, \mc{L}_{M_n'}) \leq 1/2$ (say), and then we choose $c_1$ small depending on $\alpha$ so that the above holds. Now by \eqref{prob-a-b-bd}, we have that for large enough $n$, for any $b - a \leq c_1 \sqrt{n}$, 
\[ \p(a \leq M_n \leq b) \leq \frac{1}{2} (1 + 1/4 + 1/2) = 7/8. \]
As detailed at the beginning of the proof, this implies the desired variance lower bound.
\end{proof}

\subsection{\texorpdfstring{$\lambda$-diluted minimum matching}{}}

With $\lambda$ implicit, let $\geen := \kayn(\lambda)$. Recall that $p_n = 1 - e^{-\lambda / n}$, and $F_w^{(n)}$ is the distribution of $n \text{Exp}(1)$, conditioned to lie in $[0, \lambda]$. We have $n p_n \ra \lambda$ and $d_{TV}(F_w^{(n)}, \mathrm{Unif}[0, \lambda]) \ra 0$. To verify {\localprop}, we follow the ideas of  \cite{ParWast2017, Wast2009, Wast2012}, with no claims of originality.

\begin{proof}[Verification of {\localprop}]
For $v \in \geen$, observe
\[ M_\lambda(\geen) = \min \bigg(\frac{\lambda}{2} + M_\lambda(\geen - v), \min_{u : (v, u) \in \geen} w_{(v, u)} + M_\lambda(\geen - \{v, u\})\bigg). \]
Defining $h_\lambda(\mbf{G}, v) := M_\lambda(\mbf{G}) - M_\lambda(\mbf{G} - v)$, we have
\beq\label{lambda-min-match-recurs} h_\lambda(\geen, v) = \min_{u : (v, u) \in \geen} \bigg(\frac{\lambda}{2}, w_{(v, u)} - h(\geen - \{v, u\})\bigg).\eeq
Note as $w_e \in [0, \lambda]$ for all edges $e$, we have that $h_\lambda \in [-\lambda / 2, \lambda / 2]$. Much as for maximum weight matching, we can use the recursion \eqref{lambda-min-match-recurs} to define $g^L_k, g^U_k$. And again, the key step in verifying \eqref{gu-gl-conv-0}, \eqref{gu-gl-exp-tree-conv-0} is showing that with $\btree_k \stackrel{d}{=} \btree(k, \lambda, \mathrm{Unif}[0 ,\lambda])$, we have that $g^L_k(\btree_k), g^U_k(\btree_k)$ converge in distribution to the same limit. It is essentially this condition that W\"{a}stlund calls ``replica symmetry", and it is given by Theorem 3.3 of \cite{ParWast2017} (and in a more general setting in \cite{Wast2009, Wast2012}).
\end{proof}

To verify the regularity conditions of Corollary \ref{cor-main-result}, first note $M_\lambda(\geen) - M_\lambda(\geen^e) = h_\lambda(\geen, v) - h_\lambda(\geen^e, v)$, and recall $h_\lambda(\geen, v) \in [-\lambda / 2, \lambda / 2]$. Thus we may take $H(w_e, w_e') = \lambda$. So really the only thing that needs proving is the variance lower bound.

\begin{lemma}
For fixed $\lambda > 0$, we have
\[\liminf_{n \toinf} n^{-1} \Var{M_\lambda(\geen)} > 0. \]
\end{lemma}
\begin{proof}
The proof is a small adaptation of the proof of Theorem 2.9 of \cite{Ch2018}. To follow that proof more closely, we first do some rescaling. Let $\tilde{\mbf{G}}_n$ be $\geen$ with all edge weights divided by $n$, so that the edge weights of $\tilde{\mbf{G}}_n$ are $\text{Exp}(1)$. We then consider $M_{\lambda / n}(\tilde{\mbf{G}}_n)$, which is equal to $n^{-1} M_\lambda(\geen)$. It suffices to show
\[ \liminf_{n \toinf} n \Var{M_{\lambda / n}(\tilde{\mbf{G}}_n)} > 0. \]
For brevity, denote $M_n := M_{\lambda / n}(\tilde{\mbf{G}}_n)$. As mentioned in the proof of Lemma \ref{max-match-var-lower-bd}, it suffices to find constants $c_1, c_2 > 0$ such that for large enough $n$, for $b - a \leq c_1 / \sqrt{n}$, we have
\[\p(a \leq M_n \leq b) \leq 1 - c_2. \]
Towards this end, define the function $\phi : [0, \infty) \ra [0, \infty)$,
\[
\phi(x) = 
\begin{cases} 
\sqrt{n} x & \text{if } 0 \leq x \leq 1/n \\ 
x + 1 / \sqrt{n} - 1 / n & \text{if } x > 1/n. \\ 
\end{cases} 
\]
Let $\alpha > 0$ be chosen later. Let $A = (a_{ij}, 1 \leq i < j \leq n)$ be the edge weights of $\tilde{\mbf{G}}_n$, and define $A' = (a_{ij}', 1 \leq i < j \leq n)$, where $a_{ij}'$ is such that
\[ a_{ij}' + \alpha n^{-1} \phi(a_{ij}') = a_{ij}. \]
Note as the map $x \mapsto x + \alpha n^{-1} \phi(x)$ is continuous and strictly increasing, $a_{ij}'$ exists and is unique. The proof of Theorem 2.9 of \cite{Ch2018} shows that
\[ d_{TV}(\mc{L}_{A}, \mc{L}_{A'}) \leq C \alpha. \]
Defining $M_n'$ to be the cost of the $\lambda/n$-diluted minimum matching with the weights $A'$, we thus have
\[ d_{TV}(\mc{L}_{M_n}, \mc{L}_{M_n'}) \leq C \alpha.\]
Now by Lemma 1.2 of \cite{Ch2018}, for all $-\infty < a \leq b < \infty$, we have
\beq\label{min-match-interv-bd} \p(a \leq M_n \leq b) \leq \frac{1}{2}(1 + \p(\abs{M_n - M_n'} \leq b - a) + C \alpha ), \eeq
so our goal now is to bound $\p(\abs{M_n - M_n'} \leq b - a)$.

Observe that $a_{ij}' \leq a_{ij}$ for all $i < j$, so that $M_n' \leq M_n$, so that we have $\abs{M_n - M_n'} = M_n - M_n'$. Fix $1 \geq \beta > 0$ to be chosen later. Let $b_i := \min_{j \neq i} a_{ij}$ (where $a_{ij} = a_{ji}$ if $i > j$). Let $D_n := \{i : b_i \geq \beta / n\}$. For $i \in D_n$, we have $a_{ij} \geq \beta / n$ for all $j \neq i$. As $x \mapsto x + \alpha n^{-1} \phi(x)$ is increasing, we have that $a_{ij}' \geq x_n$, where $x_n$ is the unique solution of
\[ x_n + \alpha n^{-1} \phi(x_n) = \beta n^{-1}. \]
From the definition of $\phi$, and as $\beta \leq 1$, we have
\[ x_n = \frac{\beta}{n + \alpha \sqrt{n}}. \]
Thus for $i \in D_n$, and $j \neq i$,
\begin{align*} 
a_{ij} - a_{ij}' &= \alpha n^{-1} \phi(a_{ij}') \\
&\geq \alpha n^{-1} \phi\Bigg(\frac{\beta}{n + \alpha \sqrt{n}}\Bigg) \\
&= \frac{\alpha \beta}{n^{3/2}  + \alpha n}.
\end{align*}
Now let 
\[ \begin{split}
B_n := \{&\text{vertices that are matched in the} \\ &\text{$\lambda/n$-diluted minimum matching of $\tilde{\mbf{G}}_n$} \} . \end{split}\]
We have
\begin{align*}
M_n - M_n' &\geq \frac{1}{2} \sum_{i \neq j} 1(\text{$i$ is matched to $j$ in $M_n$}) (a_{ij} - a_{ij}') \\
&\geq \frac{1}{2} \sum_{i \in B_n} \sum_{j \neq i} 1(\text{$i$ is matched to $j$ in $M_n$}) (a_{ij} - a_{ij}') \\
&\geq \frac{1}{2} \sum_{i \in D_n \cap B_n} \frac{\alpha \beta}{n^{3/2}  + \alpha n} \\
&= \frac{1}{2} \frac{\alpha \beta \abs{D_n \cap B_n}}{n^{3/2} + \alpha n}.
\end{align*}
Now suppose for the moment that $\abs{D_n \cap B_n} / n$ remains bounded away from zero in probability as $n \toinf$. Then there exists $c_1$ depending on $\alpha, \beta, \kappa$ such that
\[ \p(M_n - M_n' \leq c_1 / \sqrt{n}) \ra 0. \]
Thus recalling \eqref{min-match-interv-bd}, by taking $\alpha$ small so that $d_{TV}(\mc{L}_{M_n}, \mc{L}_{M_n'}) \leq C \alpha \leq 1/4$ (say), we have that for large enough $n$, and any $b - a \leq c_1 / \sqrt{n}$,
\[\p(a \leq M_n \leq b) \leq \frac{1}{2}(1 + 1/4 + 1/4) = \frac{3}{4}.\]
Thus the proof is complete once we show that we may take $\beta > 0$ such that $\abs{D_n \cap B_n} / n$ remains bounded away from zero in probability as $n \toinf$. As $\abs{D_n \cap B_n} \geq \abs{D_n} - \abs{B_n^c}$, this be immediate once we establish the following two lemmas.
\end{proof}

\begin{lemma}
We have $\abs{D_n} / n \stackrel{p}{\ra} \exp(-\beta)$. 
\end{lemma}
\begin{proof}
This follows by computing the first and second moments.
\end{proof}

\begin{lemma}
We have that $\abs{B_n^c} / n$ converges in probability to a constant strictly less than 1.
\end{lemma}
\begin{proof}
Note $\abs{B_n^c}$ is the number of unmatched vertices in the $\lambda$-diluted minimum matching. By Proposition 3.1 of \cite{Wast2012}, we have that $\abs{B_n^c} / n \stackrel{p}{\ra} F_\lambda(\lambda / 2)$, where $F_\lambda : [-\lambda / 2, \lambda / 2] \ra [0, 1]$ is some function. Moreover, from the proof of Proposition 2.11 of \cite{Wast2012}, we have
\[ \lambda \leq \frac{-\log F_\lambda(\lambda/2)}{F_\lambda(\lambda / 2)}, \]
which implies that $F_\lambda(\lambda /2) < 1$ for $\lambda > 0$.
\end{proof}

\section{Application of Theorem \ref{main-result} to optimal edge cover}

We devote a separate section for Optimal edge cover because unlike in Section \ref{general-theorem-applications}, we will need to spend some time establishing some basic facts before we can apply Theorem \ref{main-result}.

The approach to proving Theorem \ref{opt-edge-cover-clt} will be as follows. Because Theorem \ref{main-result} gives a rate of convergence, we will be able to first prove a central limit theorem for the quantity
\[ \frac{EC_{\lambda_n}(\kayn) - \E EC_{\lambda_n}(\kayn)}{\sqrt{\Var{EC_{\lambda_n}(\kayn)}}}, \]
where now $\lambda_n$ is taken to infinity with $n$. Moreover, we will show that $\lambda_n$ is large enough so that 
\beq\label{slutsky-1} \frac{\Var{EC(\kayn)}}{\Var{EC_{\lambda_n}(\kayn)}} \ra 1, \eeq
and
\beq\label{slutsky-2} \frac{\E EC(\kayn) - \E EC_{\lambda_n}(\kayn)}{\sqrt{\Var{EC(\kayn)}}} \ra 0. \eeq
This will then allow us to conclude Theorem \ref{opt-edge-cover-clt}.

\subsection{Basic facts of optimal edge cover}

As detailed in Section \ref{opt-edge-cover}, $EC_\lambda$ is actually a function of $\kayn(\lambda)$, and so is an optimization problem on a sparse Erd\H{o}s-R\'{e}nyi graph. However, unless the situation demands, we will continue writing $\kayn$ for brevity. 
We first investigate how large $\lambda$ needs to be for $EC_\lambda(\kayn)$ to be a good approximation of $EC(\kayn)$. We will see that the answer is $\lambda = C \log n$.

\begin{lemma}
Suppose there is a number $K$ such that every vertex $v \in \kayn$ has at least one incident edge $e$ with $w_e \leq K$. Then every edge in the optimal edge cover has weight at most $2K$.
\end{lemma}
\begin{proof}
Let $e = (v, u)$ be in the optimal edge cover. By hypothesis, there are edges $e_v, e_u$ incident to $v, u$ respectively, such that $w_{e_v}, w_{e_u} \leq K$. Now by optimality, we must have $w_e \leq w_{e_v} + w_{e_u} \leq 2K$.
\end{proof}

\begin{lemma}
If every edge in the optimal edge cover has weight at most $2K$, then we have $EC_{4K}(\kayn) = EC(\kayn)$.
\end{lemma}
\begin{proof}
As $EC_\lambda(\kayn) \leq EC(\kayn)$ for all $\lambda > 0$, only one direction needs to be proven. Let $\mc{C}$ be the optimal $4K$-diluted edge cover. If the collection $\mc{C}$ covers every vertex, then it is in fact an edge cover and thus equality is automatic. So suppose $\mc{C}$ leaves a vertex $v$ un-covered. Let $e_v$ be the edge in the optimal edge cover incident to $v$. Then by adding the edge $e_v$ to $\mc{C}$, the $\lambda$-diluted cost of $\mc{C}$ increases by at most $w_{e_v} - 2K \leq 0$. Repeating for all un-covered vertices, we obtain $EC(\kayn) \leq EC_{4K}(\kayn)$, as desired.
\end{proof}

These two lemmas show that if $\kayn(\lambda)$ has no isolated vertices, then we have $EC_{4\lambda}(\kayn) = EC(\kayn)$. Let $p_n(\lambda) := 1 - e^{-\lambda / n}$, and let $\deg_\lambda(v)$ be the degree of vertex $v$ in $\kayn(\lambda)$.

\begin{lemma}
We have
\[ \p\bigg(\deg_\lambda(v) > \frac{1}{2} n p_n(\lambda), \forall v \in V \bigg) \geq 1 - 3n e^{-n p_n(\lambda)/32}. \]
\end{lemma}
\begin{proof}
By Theorem 8.1 of \cite{FODS}, we have for a given $v \in \kayn$
\[ \p\bigg(\deg_\lambda(v) \leq \frac{1}{2} np_n(\lambda)\bigg) \leq 3 e^{-n p_n(\lambda) / 32}. \]
We conclude by applying the Union bound.
\end{proof}

\begin{prop}\label{lambda-c1-c2-large}
For any constant $C_2$, there is a constant $C_1$ possibly depending on $C_2$ such that for $\lambda_n = C_1 \log n$, and large enough $n$, we have
\[ \p(EC_{\lambda_n}(\kayn) = EC(\kayn)) \geq 1 - 3 n^{-C_2}. \]
\end{prop}
\begin{proof}
By the previous few lemmas, we have
\[ \p(EC_{\lambda_n}(\kayn) \neq EC(\kayn)) \leq \p\bigg(\exists v, \deg_\lambda(v) \leq \frac{1}{2} np_n(\lambda_n)\bigg) \leq 3n e^{-np_n(\lambda_n) / 32}.  \]
With $\lambda_n = C_1 \log n$, we have $p_n(\lambda_n) = 1 - e^{-\lambda_n / n} \geq \frac{1}{2}\frac{\lambda_n}{n}$ for large enough $n$. Thus we see that it suffices to take $C_1 = 64(C_2 + 1)$.
\end{proof}

With $\lambda_n = C_1 \log n$, we now proceed to show \eqref{slutsky-1}, \eqref{slutsky-2}. First, we need a variance lower bound.

\begin{lemma}\label{opt-edge-cover-var-lower-bd}
We have
\[ \liminf_{n \toinf} \frac{1}{n} \Var{EC(\kayn)} > 0.\]
\end{lemma}
\begin{proof}
The proof of Theorem 2.9 of \cite{Ch2018} carries over with a slight modification (the argument is for complete bipartite graphs, but it also works for complete graphs) to show
\[ \liminf_{n \toinf} n\Var{n^{-1} EC(\kayn)} > 0. \qedhere\]
\end{proof}

\begin{prop}\label{lambda-n-approx}
There is a numerical constant $C_1$ such that with $\lambda_n = C_1 \log n$, we have
\[ \frac{\Var{EC(\kayn)}}{\Var{EC_{\lambda_n}(\kayn)}} \ra 1, \]
and
\[\frac{\E EC(\kayn) - \E EC_{\lambda_n}(\kayn)}{\sqrt{\Var{EC(\kayn)}}} \ra 0. \]
\end{prop}
\begin{proof}

Take $C_2 > 100$ (say), and apply Lemma \ref{lambda-c1-c2-large} to obtain $C_1$. We first show the second assertion. Observe
\[\begin{split}
0 \leq \E& EC(\kayn) - \E EC_{\lambda_n}(\kayn) \leq \\
&(\E (EC(\kayn) - EC_{\lambda_n}(\kayn))^2)^{1/2} (\p(EC(\kayn) \neq EC_{\lambda_n}(\kayn)))^{1/2},
\end{split}
\]
where the first inequality follows since $EC_{\lambda_n}(\kayn) \leq EC(\kayn)$. Thus it follows from Lemma \ref{lambda-c1-c2-large} that for large enough $n$, we have that the right hand side above may be bounded by
\[ (\E EC(\kayn)^2)^{1/2} n^{-50} .\]
Observe that $EC(\kayn)$ is bounded by the sum of $n$ i.i.d. $n\text{Exp}(1)$ random variables, so that $\E EC(\kayn)^2 \leq C n^4$. Combining this with Lemma \ref{opt-edge-cover-var-lower-bd}, the second assertion now follows.

For the first assertion, write
\[ \begin{split}
\Var{EC_{\lambda_n}(\kayn)} = ~&\Var{EC(\kayn)} + \Var{EC(\kayn) - EC_{\lambda_n}(\kayn)} ~+ \\ &2\mathrm{Cov}(EC_{\lambda_n}(\kayn) - EC(\kayn), EC(\kayn)). 
\end{split} \]
By arguing similar to before, we may show
\[\Var{EC (\kayn) - EC_{\lambda_n}(\kayn)} \leq \E ( EC (\kayn) - EC_{\lambda_n}(\kayn))^2 \leq C n^{-46}.\]
Combining this with Lemma \ref{opt-edge-cover-var-lower-bd} and the previous observation $\E EC(\kayn)^2 \leq Cn^4$, the first assertion now follows.
\end{proof}

\subsection{Constructing the local approximation}\label{construct-local-approx}

In this section, we begin to construct the local approximations $LA_k^L, LA_k^U$ that are needed for Theorem \ref{main-result}. As before, this is done by finding a recursion for $EC_\lambda(\kayn)$. With $\lambda$ implicit, let $\geen := \kayn(\lambda)$, so that $EC_\lambda(\kayn) = EC_\lambda(\geen)$. The main difference between optimal edge cover and the problems considered in Section \ref{general-theorem-applications} is that for optimal edge cover, the recursion we derive will not be for $EC_\lambda(\geen) - EC_\lambda(\geen - v)$, and instead will be for a slightly different quantity. Indeed, this is the main reason why we can not use Corollary \ref{cor-main-result}, and instead have to resort to Theorem \ref{main-result}.

For now, fix $\lambda > 0$. What we will eventually do is apply Theorem \ref{main-result} to obtain a rate of convergence for fixed $\lambda$. This rate of convergence will be quantitative enough that we may actually take $\lambda_n= C_1 \log n$ (from Proposition \ref{lambda-n-approx}) and still have the rate converge to 0. 

For the rest of Section \ref{construct-local-approx}, we follow \cite{Wast2009}, with no claims of originality. Let $V_n$ denote the vertex set of $\geen$. For a subset of vertices $S \sse V_n$, define $EC_\lambda(\geen, S)$ to be the optimal $\lambda$-diluted edge cover of $S$, which uses edges of $\geen$. In particular, one may use edges which connect $S$ to $V_n - S$. For example, if $S$ consists of a single vertex, then $EC(\geen, S)$ will be the distance from that vertex to its nearest neighbor in $\geen$, if that distance is less than $\lambda /2$, and $\lambda / 2$ otherwise. Note also $EC_\lambda(\geen, V_n) = EC_\lambda (\geen)$. Define the function
\[ h_\lambda(v, \geen, S) := EC_\lambda(\geen, S) - EC_\lambda(\geen, S - \{v\}). \]
Observe that
\beq\label{h-lambda-interv-bound} 0 \leq h_\lambda \leq \frac{\lambda}{2}. \eeq
Now the motivation for introducing $h_\lambda$ is because for $e =(v, u)$, we may write
\beq\label{h-lambda-pert-ident}\begin{split}
EC_\lambda(\geen) - EC_\lambda(\geen^e) = ~& h_\lambda(v, \geen, V_n) - h_\lambda(v, \geen^e, V_n) ~+ \\
& h_\lambda(u, \geen, V_n - \{v\}) - h_\lambda(u, \geen^e, V_n - \{v\}) . 
\end{split} \eeq
The proof follows by noting
\[ EC_\lambda(\geen, V_n - \{v, u\}) = EC_\lambda(\geen^e, V_n - \{v, u\}), \]
because if the vertices $v, u$ do not need to be covered, then there is no need to use the edge $(v, u)$.

We now proceed to derive a recursion for $h_\lambda$, from which we will be able to construct local approximations to $h_\lambda$, and thus also to $EC_\lambda(\geen) - EC_\lambda(\geen^e)$.

\begin{lemma}
Let $v$ have neighbors $v_1, \ldots, v_d$ in $\geen$. Assume $v \in S$. We have
\[ h_\lambda(v, \geen, S) = \min_{1 \leq m \leq d} \Bigg(\frac{\lambda}{2}, w_{(v, v_m)} - h_\lambda(v_m, \geen, S - \{v\}) \Bigg).\]
\end{lemma}
\begin{proof}
The edge collection which gives $EC_\lambda(\geen, S)$ either uses at least one of the edges $(v, v_m)$, $1 \leq m \leq d$, or does not cover $v$, which incurs a cost of $\lambda/2$. Thus
\[\begin{split} 
EC_\lambda(\geen, S) = \min_{1 \leq m \leq d} &~\Bigg(\frac{\lambda}{2} + EC_\lambda(\geen, S -\{v\}), \\ 
&w_{(v, v_m)} + EC_{\lambda}(\geen, S - \{v, v_m\})\Bigg). 
\end{split}\]
Now subtract $EC_\lambda(\geen, S - \{v\})$ on both sides.
\end{proof}

Of course, for $v \notin S$, $h_\lambda(v, \geen, S) = 0$. Now by combining the above Lemma with \eqref{h-lambda-interv-bound}, we obtain the following:
\beq\label{h-lambda-recurs} h_\lambda(v, \geen, S) = \max\Bigg(0, \min_{1 \leq m \leq d} \Bigg(\frac{\lambda}{2}, w_{(v, v_m)} - h_\lambda(v_m, \geen, S - \{v\})\Bigg)\Bigg), \eeq
\beq\label{h-lambda-nn-bound} h_\lambda(v, \geen, S) \leq \min \Bigg(\frac{\lambda}{2}, \min_{1 \leq m \leq d} w_{(v, v_m)} \Bigg). \eeq
We now use the recursion \eqref{h-lambda-recurs} to construct the local approximations. With $\lambda$ implicit, we define functions $h^L_k, h^U_k$. Let $\btree$ be a rooted weighted tree of depth at most $k$. For vertices $v \in \btree$ at depth $k$, define
\[ h^L_k(v; \btree) := 0, ~h^U_k(v; \btree) := \frac{\lambda}{2} \]
if $k$ is even, and
\[ h^L_k(v; \btree) := \frac{\lambda}{2},~ h^U_k(v; \btree) := 0 \]
if $k$ is odd. For leaf vertices $v \in \btree$ at depth less than $k$, define
\[ h^L_k(v; \btree) = h^U_k(v; \btree) := \frac{\lambda}{2}. \]
For non-leaf vertices $v \in \btree$, define
\[ h^L_k(v; \btree) := \max\Bigg(0, \min_{u \in \mc{C}(v)} \Bigg(\frac{\lambda}{2}, \ell_u - h^L_k(u; \btree) \Bigg)\Bigg), \]
and
\[ h^U_k(v; \btree) := \max\Bigg(0, \min_{u \in \mc{C}(v)} \Bigg(\frac{\lambda}{2}, \ell_u - h^U_k(u; \btree) \Bigg)\Bigg). \]
Observe in particular we have
\beq\label{hlk-huk-easy-bound} 0 \leq h^L_k, h^U_k \leq \lambda / 2 \eeq

\begin{lemma}\label{h-lambda-bracket}
Suppose $B_k := B_k(v, \geen)$ is a tree. We have
\[h^L_k(v; B_k) \leq h_\lambda(v, \geen, V_n) \leq h^U_k(v; B_k). \]
Moreover, for any $u$ connected to $v$ in $\geen$, we have
\[ \begin{split}
\min(h^U_k(u; B_k), w_{(v, u)}) \leq h_\lambda(u, \geen, V_n - \{v\}) \leq \min(h^L_k(u; B_k), w_{(v, u)}). 
\end{split}\]
\end{lemma}
\begin{proof}
By \eqref{h-lambda-recurs}, the first inequality follows from the second inequality. The second inequality follows by induction on $k$.
\end{proof}

With $h^L_k, h^U_k$ defined, we could proceed (using \eqref{h-lambda-pert-ident}) to define the local approximations $LA_k^L, LA_k^U$. However, we decide to delay this to Section \ref{completing-proof-edge-cover}.

\subsection{Quantitative bound for the error in the local approximation}

To apply Theorem \ref{main-result}, we need to bound the error in local approximation (i.e. $\delta_k$ in (A3) of {\localprop}). With $\lambda$ implicit, let $\btree_\infty \stackrel{d}{=} \mbf{T}(\infty, \lambda, \text{Exp}(1))$, and let $\mbf{T}_k$ be the depth $k$ subtree of $\mbf{T}_\infty$. We write $h^L_k(\troot) := h^L_k(\troot; \btree_k)$, $h^U_k(\troot) := h^U_k(\troot; \btree_k)$ for brevity. One of the main terms in the error turns out to be $\E (h^U_k(\troot) - h^L_k(\troot))^2$.
Now the results of \cite{Wast2009} immediately imply that $\lim_{k \toinf} \E (h^U_k(\troot) - h^L_k(\troot))^2 = 0$ for fixed $\lambda > 0$, but we will need a more quantitative bound, due to the fact that we are trying to take $\lambda \toinf$ with $n$.

An inductive argument shows $h^L_k(\troot) \leq h^U_k(\troot)$ for all $k$, which implies
\beq\label{hbk-hak-sq-bd} \E (h^U_k(\troot) - h^L_k(\troot))^2 \leq \E (h^U_k(\troot))^2 - \E (h^L_k(\troot))^2, \eeq
which we will use later. 


\begin{prop}\label{decay-of-corr}
For $\lambda > 0$, we have
\[ \E (h^U_k(\troot) - h^L_k(\troot))^2 \leq C \lambda \alpha(\lambda)^k. \]
Here $\alpha(\lambda) < 1$, and even more, $\sup_{\lambda \geq \delta} \alpha(\lambda) < 1$ for all $\delta > 0$.
\end{prop}

\begin{remark}
The immediate consequence of this proposition is that if we take $\lambda_n = C_1 \log n$ (as in Proposition \ref{lambda-n-approx}), then upon taking $k_n = C_1' \log \lambda_n$ for some large enough $C_1'$, we have that $\E (h^U_{k_n}(\troot) - h^L_{k_n}(\troot))^2 \ra 0$.

One may think of this as an exponential delay of correlations result. In particular, $h^L_k(\troot), h^U_k(\troot)$ are defined by setting some initial conditions at the leaf vertices of $\mbf{T}_k$, and then recursively defining the values of $h^L_k, h^U_k$ for all non-leaf vertices. This proposition is essentially saying that the effect of the initial conditions is swept away exponentially quickly in the depth of the tree $\mbf{T}_k$. 
\end{remark}

To prove Proposition \ref{decay-of-corr}, we first need to establish the following relation between the distributions of $h^L_k, h^U_k$. For $\lambda > 0$, define the operator $V_\lambda$ on functions $F : [0, \lambda / 2] \ra [0, 1]$ as follows:
\[ (V_\lambda F)(x) := \exp(-\int_0^{\lambda / 2} F(\ell) d\ell) e^{-x}. \]
For notational purposes, define $E_\lambda(F) := \int_0^{\lambda / 2} F(\ell) d\ell$, so that $(V_\lambda F)(x) = e^{-E_\lambda(F)} e^{-x}$. For $k \geq 0$, $x \in [0, \lambda / 2]$, let
\[ F_k(x) := \p(h^L_k(\troot) \geq x), ~~ G_k(x) := \p(h^U_k(\troot) \geq x). \]

\begin{lemma}
We have
\[ F_{k+1} = V_\lambda G_k, \]
and
\[ G_{k+1} = V_\lambda F_k. \]
\end{lemma}

This result is implicit in Section 4 of \cite{Wast2009}. The proof takes advantage of the recursive properties of $\btree_k$, as well as the fact that the offpsring distribution is $\text{Poisson}(\lambda)$, to reduce to a Poisson process calculation. A detailed proof in the case of $\lambda$-diluted minimum matching (with more general edge cost distribution) is given in Section 2.7 of \cite{Wast2012}.

We now collect several simple facts about the operator $V_\lambda$.

\begin{lemma}
A fixed point of $V_\lambda$ is the function
\[F_\lambda (x) := e^{-A_\lambda} e^{-x}, \]
where $A_\lambda = E_\lambda (F_\lambda)$ satisfies
\[ A_\lambda = e^{-A_\lambda} (1 - e^{-\lambda / 2}). \]
\end{lemma}

\begin{lemma}\label{v-lambda-anti-mono}
For functions $F \leq G$, we have $V_\lambda F \geq V_\lambda G$. As $F_0 \leq F_\lambda \leq G_0$, we then have $F_k \leq F_\lambda \leq G_k$ for all $k$. Moreover, $F_k \leq F_{k+1}$, $G_{k+1} \leq G_k$ for all $k$.
\end{lemma}

We now analyze the operator $V_\lambda$. In light of \eqref{hbk-hak-sq-bd}, Proposition \ref{decay-of-corr} is a consequence of the following slightly more general proposition.

\begin{prop}\label{v-lambda-exp-conv}
We have for $x \in [0, \lambda/2]$,
\[ \abs{F_\lambda(x) - F_k(x)} \leq C\lambda \alpha(\lambda)^k e^{-x}, \]
and
\[ \abs{G_k(x) - F_\lambda(x)} \leq C \alpha(\lambda)^k e^{-x}, \]
where $\alpha(\lambda) < 1$ for all $\lambda > 0$, and even more, $\sup_{\lambda \geq \delta} \alpha(\lambda) < 1$ for all $\delta > 0$.
\end{prop}
\begin{proof}
To start, observe
\begin{align*} 
G_{k+2}(x) - F_\lambda(x) &= e^{- E_\lambda(F_{k+1}(x))} e^{-x} - e^{-A_\lambda} e^{-x} \\
&= e^{-(x + E_\lambda(F_{k+1}))} \bigg(1 - e^{-(A_\lambda - E_\lambda(F_{k+1}))}\bigg).
\intertext{As $F_{k+1} \leq F_\lambda$ by Lemma \ref{v-lambda-anti-mono}, we have $E_\lambda(F_{k+1}) \leq E_\lambda(F_\lambda) = A_\lambda$, and thus}
&\leq e^{-(x + E_\lambda(F_{k+1}))} (A_\lambda - E_\lambda(F_{k+1})).
\end{align*}
Integrating over $0 \leq x \leq \lambda / 2$, we obtain
\[ E_\lambda(G_{k+2}) - A_\lambda \leq e^{-E_\lambda(F_{k+1})} (A_\lambda - E_\lambda(F_{k+1})).\]
The same argument also shows
\[ A_\lambda - E_\lambda(F_{k+1}) \leq e^{-A_\lambda} (E_\lambda(G_k) - A_\lambda).  \]
Now as $E_\lambda(F_{k+1}) \geq E_\lambda(F_0) = 0$, combining the above two displays, we have
\[ E_\lambda(G_{k+2}) - A_\lambda \leq e^{-A_\lambda} (E_\lambda(G_k) - A_\lambda).\]
The same argument implies
\[ A_\lambda - E_\lambda(F_{k+2}) \leq e^{-A_\lambda} (A_\lambda - E_\lambda(F_k)). \]
Iterating these inequalities, we obtain
\[ E_\lambda(G_{2k}) - A_\lambda \leq e^{-kA_\lambda} (E_\lambda(G_0) - A_\lambda) \leq \lambda e^{-kA_\lambda}, \]
and
\[ A_\lambda - E_\lambda F_{2k} \leq e^{-kA_\lambda} (A_\lambda - E_\lambda(F_0)) = A_\lambda e^{-kA_\lambda}.\]
As $F_k \leq F_{k+1}$, and $G_{k+1} \leq G_k$, we obtain
\[ E_\lambda(G_k) - A_\lambda \leq \lambda e^{-\floor{k/2} A_\lambda}, \]
\[ A_\lambda - E_\lambda(F_k) \leq A_\lambda e^{-\floor{k/2} A_\lambda}. \]
Substituting back into our previously derived inequalities, and using Lemma \ref{v-lambda-anti-mono}, we obtain
\[ 0 \leq G_{k+1}(x) - F_\lambda(x) \leq A_\lambda e^{-\floor{k/2} A_\lambda} e^{-x}, \]
\[ 0 \leq F_\lambda(x) - F_{k+1}(x) \leq \lambda A_\lambda e^{-\floor{k/2} A_\lambda} e^{-x}. \]
To finish, observe that by using the definition of $A_\lambda$ and the intermediate value theorem, we have that $A_\lambda$ is increasing in $\lambda$. Thus for any $\delta > 0$, we have $\inf_{\lambda \geq \delta} A_\lambda = A_\delta > 0$. Observe also that $A_\lambda \leq A_\infty$, where $A_\infty$ satisfies $A_\infty = e^{-A_\infty}$.
\end{proof}

\subsection{Completing the proof of the central limit theorem}\label{completing-proof-edge-cover}

We now have all the pieces in place to deduce Theorem \ref{opt-edge-cover-clt} from Theorem \ref{main-result}. We first define the local approximations $LA_k^L, LA_k^U$. The main idea is to use \eqref{h-lambda-pert-ident}, and approximate $h_\lambda$ by $h^L_k$, $h^U_k$. Fix $e = (v, u)$, and let $B_k := B_k(v, \geen)$, $B_k' := B_k(v, \geen^e)$ for brevity. When $B_k, B_k'$ are both trees, we define
\[\begin{split}
LA_k^L(B_k, B_k')& := h^L_k(v; B_k) - h^U_k(v; B_k') ~+ \\
&1(b_e = 1, b_e' = 0) \bigg(\min(h^U_k(u; B_k), w_e) - h^L_k(u; B_k)\bigg) ~+\\
& 1(b_e = 0, b_e' = 1) \bigg(h^U_k(u; B_k') - \min(h^L_k(u; B_k'), w_e')\bigg) ~+ \\
&1(b_e = 1, b_e' = 1) \bigg(\min(h^U_k(u; B_k), w_e) - \min(h^L_k(u; B_k'), w_e') \bigg).
\end{split} \]
The function $LA_k^U$ is defined similarly, by swapping the roles of $h^L_k, h^U_k$. One may verify (A1) by using Lemma \ref{h-lambda-bracket}. By \eqref{h-lambda-interv-bound}, \eqref{h-lambda-pert-ident}, we may take $H(w_e, w_e') := 2\lambda$, so that $J = 64 \lambda^6$. 

The following lemma gives quantitative bounds on the numbers $(\delta_k, k \geq 1)$ which appear in {\localprop}.

\begin{lemma}\label{opt-edge-cover-local-approx-error-bound}
Let $\delta_k(\lambda)$ be defined as in (A3) of {\localprop}, for the optimization problem $EC_\lambda$. Then
\[ \delta_k(\lambda) \leq C \lambda \alpha(\lambda)^{k-1},\]
with $\alpha(\lambda) < 1$, and even more, $\sup_{\lambda \geq c} \alpha(\lambda) < 1$ for all $c > 0$.
\end{lemma}

Before we prove this lemma, we first show how Theorem \ref{opt-edge-cover-clt} follows.

\begin{proof}[Proof of Theorem \ref{opt-edge-cover-clt}]
Observe $p_n = 1 - e^{-\lambda / n} \leq \lambda / n$. To bound $d_{TV}(F_w^{(n)}, F_w)$, one may upper bound the $L^1$ distance between the densities of $F_w^{(n)}, F_w$. Here $F_w^{(n)}$ is the distribution of $n \text{Exp}(1)$ conditioned to lie in $[0, \lambda]$, and $F_w$ is $\mathrm{Unif}[0, \lambda]$. A calculation shows that the $L^1$ distance may be bounded by $C\lambda / n$. We thus have for fixed $\lambda$, the term $\varep_k(n)$ from Theorem \ref{main-result} may be bounded
\[ \varep_k(n) \leq \frac{C^k(\lambda + C)^{k+C}}{n^{1/3} \min(\lambda, 1)}. \]

Let $(\sigma_n^\lambda)^2 := \Var{EC_\lambda(\kayn)}$, $r_n^\lambda := n / (\sigma_n^\lambda)^2$, and
\[ Z_n^\lambda := \frac{EC_\lambda(\kayn) - \E EC_\lambda(\kayn)}{\sigma_n^\lambda} = \frac{EC_\lambda(\mbf{G}_n) - \E EC_\lambda(\mbf{G}_n)}{\sigma_n^\lambda}. \]
Using Lemma \ref{opt-edge-cover-local-approx-error-bound}, upon applying Theorem \ref{main-result} we obtain
\[\begin{split}
&\sup_{t \in \R} \abs{ \p(Z_n^\lambda \leq t) - \Phi(t)} \leq \\
&C (r_n^\lambda)^{1/2} \bigg[ \lambda^C \alpha(\lambda)^{(k-1)/8} + \frac{C^k(\lambda + C)^{k+C}}{n^{1/C} \min(\lambda, 1)} \bigg] +
C(r_n^\lambda)^{3/4} \frac{\lambda^C}{n^{1/4}}.
\end{split}\]
Now take $\lambda_n = C_1 \log n$ as in Proposition \ref{lambda-n-approx}, and take $k_n = C_1' \log \lambda_n$ for some $C_1'$ large enough depending on $C_1$, such that
\[ \lim_{n \toinf} \lambda_n^C \alpha(\lambda_n)^{(k_n - 1)/8} = 0.\]
Note as $k_n$ grows like $\log \log n$, we have that
\[ \lim_{n \toinf} \frac{C^{k_n}(\lambda_n + C)^{k_n + C}}{n^{1/C}} = 0.\]
Finally, by Lemma \ref{opt-edge-cover-var-lower-bd} and Proposition \ref{lambda-n-approx}, we have that 
\[ \limsup_n r_n^{\lambda_n} < \infty.\]
Upon combining these observations, we obtain
\[\lim_{n \toinf} \sup_{t \in \R} \abs{\p(Z_n^{\lambda_n} \leq t) - \Phi(t)} = 0. \]
Thus $Z_n^{\lambda_n} \stackrel{d}{\lra} N(0, 1)$, and by Proposition \ref{lambda-n-approx}, we can conclude
\[\frac{EC(\kayn) - \E EC(\kayn)}{\sqrt{\Var{\kayn}}} \stackrel{d}{\lra} N(0, 1). \qedhere \]
\end{proof}

\begin{proof}[Proof of Lemma \ref{opt-edge-cover-local-approx-error-bound}]
Recall the definitions of $\btree_k, \tilde{\btree}_k$ in (A3) of {\localprop}. We will show how to obtain the bound for the pair $(\tilde{\btree}_k, \btree_k)$. The case of $(\btree_k, \tilde{\btree}_k)$ will have the exact same proof.

Recall $\tilde{\btree}_k$ is constructed from $(\btree_k, \btree_{k-1}', \troot, \troot', \ell)$. Thus we have (noting that $h^L_k(\troot', \tilde{\btree}_k) = h^U_{k-1}(\troot', \btree_{k-1}')$)
\[ \begin{split}
LA_k^U(\tilde{\btree}_k, \btree_k) = h^U_k(\troot; \tilde{\btree}_k) &- h^L_k(\troot; \btree_k) ~+ \\
&\min(h^U_{k-1}(\troot', \btree_{k-1}'), \ell) - h^L_{k-1}(\troot', \btree_{k-1}').
\end{split}\]
Similarly,
\[ \begin{split}
LA_k^L(\btree, \btree') = h^L_k(\troot; \tilde{\btree}_k) &- h^U_k(\troot; \btree_k) ~+ \\
&\min(h^L_{k-1}(\troot', \btree_{k-1}'), \ell) - h^U_{k-1}(\troot', \btree_{k-1}').
\end{split} \]
Thus
\[ \begin{split}
(LA_k^U(\btree, \btree') - LA_k^L(\btree, \btree'))^2 \leq ~&C (h^U_k(\troot; \tilde{\btree}_k) - h^L_k(\troot; \tilde{\btree}_k))^2 +  ~\\
& C (h^U_k(\troot; \btree_k) - h^L_k(\troot; \btree_k))^2 ~+ \\
& C (h^U_{k-1}(\troot', \btree_{k-1}') - h^L_{k-1}(\troot', \btree_{k-1}'))^2.
\end{split}\]
Upon taking expectations, the last two terms in the right hand side above may be handled by Proposition \ref{opt-edge-cover-local-approx-error-bound}. To handle the first term, observe
\[ h^U_k(\troot, \tilde{\btree}_k) = \max\bigg(0, \min\bigg(h^U_k(\troot, \btree_k), \ell - h^L_{k-1}(\troot', \btree_{k-1}') \bigg)\bigg), \]
and similarly,
\[ h^L_k(\troot, \tilde{\btree}_k) = \max\bigg(0, \min\bigg(h^L_k(\troot, \btree_k), \ell - h^U_{k-1}(\troot', \btree_{k-1}')\bigg)\bigg).\]
Thus
\[\begin{split} 
(h^U_k(\troot, \tilde{\btree}_k) - h^L_k(\troot, \tilde{\btree}_k))^2 \leq~& (h^U_k(\troot, \btree_k) - h^L_k(\troot, \btree_k))^2 ~+\\
&(h^U_{k-1}(\troot', \btree_{k-1}') - h^L_{k-1}(\troot', \btree_{k-1}'))^2.
\end{split}\]
Upon applying Proposition \ref{opt-edge-cover-local-approx-error-bound}, we obtain
\[ \E (h^U_k(\troot, \tilde{\btree}_k) - h^L_k(\troot, \tilde{\btree}_k))^2 \leq C \lambda \alpha(\lambda)^{k-1}. \]
Collecting the previous results allows us to obtain
\[ \E (LA^U_k(\tilde{\btree}_k, \btree_k) - LA^L_k(\tilde{\btree}_k, \btree_k))^2 \leq C \lambda \alpha(\lambda)^{k-1},\]
as desired.
\end{proof}

\section{Proofs}\label{proofs-section}

In this section, we prove Corollary \ref{cor-main-result} and Theorem \ref{main-result}. The proof of Theorem \ref{main-result} will rely on certain facts about neighborhoods of Erd\H{o}s-R\'{e}nyi graphs, which are covered in Section \ref{er-graph-facts}.

\subsection{Proof of Corollary \texorpdfstring{\ref{cor-main-result}}{}}

Let $e = (v, u)$. When $B_k := B_k(v, \geen)$, $B_k' := B_k(v, \geen^e)$ are trees, define
\[ LA_k^L(B_k, B_k') := g^L_k(B_k) - g^U_k(B_k'), \]
and
\[ LA_k^U(B_k, B_k') := g^U_k(B_k) - g^L_k(B_k'). \]
We proceed to verify {\localprop}. (A1) follows by \eqref{add-vertex-bracket}. To show (A2), note if $(B_k, B_k') \cong (\btree, \btree')$, then $B_k \cong \btree$, $B_k' \cong \btree'$, and thus by \eqref{gk-label-invariant}, we have
\[ LA_k^L(\btree, \btree') =  g^L_k(\btree) - g^U_k(\btree') = g^L_k(B_k) - g^U_k(B_k') = LA_k^L(B_k, B_k'), \]
and similarly for $LA_k^U$. Finally, to show (A3), observe that for trees $\btree, \btree'$, we have
\[ |LA_k^U(\btree, \btree') - LA_k^L(\btree, \btree')| \leq |g^U_k(\btree) - g^L_k(\btree)| + |g^U_k(\btree') - g^L_k(\btree')|.\]
(A3) now follows by \eqref{gu-gl-conv-0}, \eqref{gu-gl-exp-tree-conv-0}. 

For $\varep_k(n)$ as defined in Theorem \ref{main-result}, since $np_n \ra \gendeg$ and $d_{TV}(F_w^{(n)}, F_w) \ra 0$, we have that for all $k$, $\lim_{n \toinf} \varep_k(n) = 0$. Upon applying Theorem \ref{main-result}, because we've assumed the variance lower bound, we obtain for all $k > 0$
\[ \limsup_{n \toinf} \sup_{t \in \R} \abs{\p(Z_n \leq t) - \Phi(t)} \leq C' \delta_k^{1/8}, \]
where $C'$ is some finite number which may depend on $(f, (\geen, n \geq 1))$, but not $k$, and $\delta_k$ is as in {\localprop}. To finish, take $k \toinf$. \qedsymbol

\subsection{Proof of Theorem \texorpdfstring{\ref{main-result}}{}}

Let $E_n := \{(i, j) : 1 \leq i < j \leq n\}$. Given $e = (i, j)$, let $\geen - e$ be the weighted graph obtained by deleting edge $e$ if it is present, else doing nothing. Similarly define $\geen - F$ for a set of edges $F \sse E_n$. Let $(W', B')$ be an independent copy of $(W, B)$. Recalling the definition of $\geen^e$, for a set $F \sse  E_n$, define $\geen^F$ to be the weighted graph obtained by using $w_e', b_e'$ instead of $w_e, b_e$ for $e \in F$.
For singleton sets $\{e\}$, we will write $\geen^{e}$ instead of $\geen^{\{e\}}$, and we will write $\geen^{F \cup e}$ instead of $\geen^{F \cup \{e\}}$. Recalling the definition of $\Delta_e f := f(\geen) - f(\geen^e)$, for $F \sse E_n \backslash \{e\}$ similarly let 
\[ \Delta_e f^F := f(\geen^F)  - f(\geen^{F \cup e}). \]
Recall $\sigma_n^2 := \mathrm{Var}(f(\geen))$, $\gendeg_n := n p_n$. The following is Corollary 3.2 of \cite{Ch2014}, adapted to our situation.

\begin{lemma}\label{gen-pert-normal-approx}
For each $e = (v, u)$, $e' = (v', u')$, let $c(e, e')$ be such that for all $F \sse E_n \backslash \{e\}$, $F' \sse E_n \backslash \{e'\}$, we have
\[ \frac{1}{\sigma_n^4} \mathrm{Cov}(\Delta_{e} f \Delta_{e} f^{F}, \Delta_{e'} f \Delta_{e'} f^{F'}) \leq c(e, e'). \]
Then
\[ \begin{split} 
\sup_{t \in \R} \abs{\p(Z_n \leq t) - \Phi(t)} \leq \sqrt{2} \Bigg(&\sum_{e, e' \in E_n} c(e, e') \Bigg)^{1/4} + \Bigg(\frac{1}{\sigma_n^3} \sum_{e \in E_n} \E \abs{\Delta_e f}^3 \Bigg)^{1/2}. 
\end{split}\]
\end{lemma}

\begin{remark}
As we will see, the only terms that are nontrivial to bound are $c(e, e')$ for edges $e, e'$ with distinct vertices. This is Lemma \ref{c-e1e2-bound}, but the main work is done by Lemmas \ref{l-k-cov-bound} and \ref{remainder-lemma}. All other terms will be bounded by applying the assumed regularity conditions and using Cauchy-Schwarz or related inequalities.
\end{remark}

We begin by bounding the second term.

\begin{lemma}\label{second-term-bound}
We have
\[ \sum_{e \in E_n} \E \abs{\Delta_e f}^3 \leq  J^{1/2} n^2 p_n = J^{1/2} n\gendeg_n. \]
\end{lemma}
\begin{proof}
This follows by conditions \eqref{f-regularity} and \eqref{H-moment}, the independence of $W, B$, and a calculation.
\end{proof}

\begin{lemma}\label{c-ee-bound}
For any edge $e$, we may take
\[ c(e, e) = \frac{1}{\sigma_n^4} C J^{2/3} p_n = \frac{1}{\sigma_n^4} \frac{C J^{2/3} \gendeg_n}{n}. \]
\end{lemma}
\begin{proof}
For $F, F' \sse E_n \backslash \{e\}$, we want to bound
\[ \mathrm{Cov}(\Delta_e f \Delta_e f^F, \Delta_e f \Delta_e f^{F'}). \]
This may be done by applying \eqref{f-regularity} and \eqref{H-moment}.
\end{proof}

\begin{lemma}\label{share-one-vertex-bound}
Let $e = (v, u), e' = (v', u') \in E_n$ be edges which share exactly one vertex. We may take
\[ c(e, e') = \frac{1}{\sigma_n^4} C J^{2/3} p_n^2 = \frac{1}{\sigma_n^4} \frac{C J^{2/3} \gendeg_n^2}{n^2}. \]
\end{lemma}
\begin{proof}
Again, we apply \eqref{f-regularity} and \eqref{H-moment}, and proceed.
\end{proof}

The main work will be in bounding $c(e, e')$ for edges $e = (v, u)$, $e' = (v', u')$ with all distinct vertices. Let $E_0$ be the event that both $B_k(v, \geen), B_k(v, \geen^e)$ are trees. Let $A_e := \{(b_e, b_e') = (1, 0) \text{ or } (0, 1)\}$. Define
\[ \tilde{L}_k^{e} := 1_{E_0}1_{A_e} LA^L_k(B_k(v,  \geen), B_k(v, \geen^{e})). \]
Let $Q(x, y) := \max(\min(x, y), -y)$ (in words, $Q$ is truncation of $x$ at level $y$). Define
\[ L_k^e := Q(\tilde{L}_k^e, H(w_e, w_e')), \]
($``L"$ is for ``local"), so that
\beq\label{l-k-regularity} \abs{L_k^e} \leq 1_{A_e} H(w_e, w_e'). \eeq
Let
\[R_k^e := \Delta_e f - L_k^e. \]
Here $``R"$ is for ``remainder". Let $\tilde{A}_e := \{\max(b_e, b_e') = 1\}$, so that $\tilde{A}_e = A_e \cup \{b_e, b_e' = 1\}$. Observe by \eqref{f-regularity}, \eqref{l-k-regularity}, we have
\beq\label{r-k-regularity} \abs{R_k^e} \leq 21_{\tilde{A}_e} H(w_e, w_e'). \eeq

For $F \sse E_n \backslash \{e\}$, we may define $L_k^{F \cup e}$ by using $B_k(v, \geen^F), B_k(v, \geen^{F \cup e})$ in place of $B_k(v, \geen)$, $B_k(v, \geen^e)$. Then let
\[ R_k^{F \cup e} := \Delta_e f^F - L_k^{F \cup e}. \]
To bound $c(e, e')$, we need to upper bound
\beq\label{cov-lk-uk} \mathrm{Cov}\bigg( (R_k^{e}+ L_k^{e})(R_k^{F \cup e} + L_k^{F \cup e}), (R_k^{e'} + L_k^{e'})(R_k^{F' \cup e'} + L_k^{F' \cup e'}) \bigg) .\eeq
Upon expanding this covariance, we obtain 16 terms, one of which only involves local approximation quantities. We should think of local quantities as essentially independent, and so their covariance should be essentially 0. The following lemma makes this precise.

\begin{lemma}\label{l-k-cov-bound}
With $\rho_k(n)$ as in Theorem \ref{main-result}, we have
\[ \mathrm{Cov}(L_k^{e} L_k^{F \cup e}, L_k^{e'} L_k^{F' \cup e'}) \leq \locallemmabound. \]
\end{lemma}

To not distract too much from the main thrust of the argument, we will defer the proof of this lemma to Section \ref{lk-loc-proof}. The other 15 terms which come from expanding \eqref{cov-lk-uk} all involve at least one remainder term.

\begin{lemma}\label{remainder-lemma}
With $\delta_k$ as in {\localprop} and $\varep_k(n), \rho_k(n)$ as in Theorem \ref{main-result}, any of the other 15 terms which come from expanding \eqref{cov-lk-uk} may be bounded by
 \[ \remainderbound.\]
\end{lemma}
\begin{proof}
Note $J \geq 1$ by definition, so we may bound $J^{r} \leq J$, for $r \leq 1$. We will show how to bound a term like
\[ \mathrm{Cov}(R_k^{e} X_2, X_3 X_4), \]
where $X_2$ is either $R_k^{F \cup e}$ or $L_k^{F \cup e}$, $X_3$ is either $R_k^{e'}$ or $L_k^{e'}$, and $X_4$ is either $R_k^{F' \cup e'}$ or $L_k^{F' \cup e'}$. The other terms may be bounded in a similar manner. Define 
\[ \tilde{H} := \max(H(w_e, w_e'), H(w_{e'}, w_{e'}')). \]
By \eqref{l-k-regularity}, \eqref{r-k-regularity}, we have
\begin{align*} 
\abs{R_k^e X_2 X_3 X_4} &\leq C1_{\tilde{A}_e} 1_{\tilde{A}_{e'}} \tilde{H}^3 \abs{R_k^e} \\
&\leq C 1_{A_e} 1_{\tilde{A}_{e'}} \tilde{H}^3 \abs{R_k^e} + C 1(b_e, b_e' = 1) 1_{\tilde{A}_{e'}} \tilde{H}^4.
\end{align*}
By the independence of $W, B, B'$, we have
\[\E [1(b_e, b_e' = 1) 1_{\tilde{A}_{e'}} \tilde{H}^4] \leq C \E \tilde{H}^4 p_n^3 \leq C J^{2/3} p_n^3. \]
As $p_n = \lambda_n / n \leq \rho_k(n)$, we are done with this term. Moving on to the other term, let $Y_e := (b_e, b_e')$. By Cauchy-Schwarz and the independence of $W, B$, we have
\[ \E [\tilde{H}^3 \abs{R_k^e} ~|~ Y_e, Y_{e'}] \leq CJ^{1/2} (\E [(R_k^e)^2 ~|~ Y_e, Y_{e'}])^{1/2}.\]
For brevity, write $B_k, B_k'$ instead of $B_k(v, \geen)$, $B_k(v, \geen^e)$. By (A1), \eqref{l-k-regularity}, we have
\[ 1_{A_e}\abs{R_k^e} \leq 1_{E_0}1_{A_e} (LA_k^U(B_k, B_k') - LA_k^L(B_k, B_k')) + C \tilde{H} 1_{E_0^c}.  \]
Thus
\beq\label{rke-cond-exp} \begin{split}
1_{A_e}\E [ (R_k^e)^2 ~|~ Y_e, Y_{e'}] \leq  21_{A_e}\E[ 1_{E_0}& (LA_k^U(B_k, B_k') - LA_k^L(B_k, B_k'))^2 ~|~ Y_e, Y_{e'}] \\
&+ C J^{1/3} (\p(E_0^c ~|~ Y_e, Y_{e'}))^{1/2}. 
\end{split}\eeq
We may bound
\[ \p(E_0^c ~|~ Y_e, Y_{e'}) \leq \p(B_k \text{ not a tree} ~|~ b_e, b_{e'}) + \p(B_k' \text{ not a tree} ~|~ b_e', b_{e'}). \]
We may bound
\[ \p(B_k \text{ not a tree} ~|~ b_e, b_{e'}) \leq \min\bigg(\frac{(\gendeg_n + C)^{2k+C}}{n}, 1\bigg), \]
and similarly for $B_k'$. This may be done by noting that if $B_k$ is not a tree, then either $B_k(v, \geen - e)$ is not a tree, or $B_{k-1}(u, \geen - e)$ is not a tree, or $B_k(v, \geen - e), B_{k-1}(u, \geen - e)$ intersect. To remove the conditioning, we may show that $B_k(v, \geen - e) = B_k(v, \geen - \{e, e'\})$ and $B_{k-1}(u, \geen - e) = B_{k-1}(u, \geen - \{e, e'\})$ with very high probability, even conditional on $b_e, b_{e'}$. Then finish by Lemma \ref{nbd-tree-prob}.

Moving to bound the other term in \eqref{rke-cond-exp}, we apply Lemma \ref{nbd-perturb-joint-coupling} to couple $(B_k, B_k', \btree, \btree')$. Observe by Lemma \ref{nbd-joint-coupling}, we may take $\varep_k(n)$ defined in Lemma \ref{nbd-perturb-joint-coupling} to be exactly the $\varep_k(n)$ that is given in the statement of Theorem \ref{main-result}. Let $E_1 := \{(B_k, B_k') \cong (\btree, \btree')\}$. By (A2), (A3) of {\localprop}, we have
\[\begin{split} 
1_{A_e}\E[ 1_{E_0} (LA_k^U&(B_k, B_k') - LA_k^L(B_k, B_k'))^2 ~|~ Y_e, Y_{e'}] \leq \delta_k ~+ \\
&C J^{1/3} \bigg(\jointcouplingprob\bigg)^{1/2}.  
\end{split}\]
We have thus bounded $\E \abs{R_k^e X_2 X_3 X_4}$. Note the term $d_{TV}(F_w^{(n)}, F_w)$ may be absorbed into $\varep_k(n)$, and $(\lambda_n + 1)^{2k} / n$ may be bounded by $\rho_k(n)$ (we may take min of $(\lambda_n + 1)^{2k} / n$ with 1 since this term comes from bounding a probability). The term $\E R_k^e X_2 \E X_3 X_4$ may be similarly bounded.
\end{proof}

We now collect Lemmas \ref{l-k-cov-bound}, \ref{remainder-lemma} into the following lemma. Here we also use the fact that by definition, $\rho_k(n) \leq 1$, so that $\rho_k(n) \leq \rho_k(n)^{1/4}$.

\begin{lemma}\label{c-e1e2-bound}
For edges $e, e'$ using all distinct vertices, we may take
\[ \begin{split}
 c(e, e') = \frac{C J p_n^2}{\sigma_n^4} \bigg(\delta_k^{1/2} + \varep_k(n)^{1/4} + \rho_k(n)^{1/4}\bigg).
 \end{split}\]
\end{lemma}

\begin{proof}[Proof of Theorem \eqref{main-result}]
Combine Lemmas \eqref{gen-pert-normal-approx}, \eqref{second-term-bound}, \eqref{c-ee-bound}, \eqref{share-one-vertex-bound}, \eqref{c-e1e2-bound}.
\end{proof}

\subsection{Proof of Lemma \ref{l-k-cov-bound}}\label{lk-loc-proof}

First, we set some notation. Define $S_e := (w_e, b_e, w_e', b_e')$, $A_e := \{(b_e, b_e') = (1, 0) \text{ or } (0, 1)\}$. Let $X_e := L_k^e L_k^{F \cup e}$, and $X_{e'} := L_k^{e'} L_k^{F' \cup e'}$. By definition of $L_k^e$, note that $X_e = 1_{A_e} X_e$, and similarly for $X_{e'}$. We may write
\[\begin{split}
\mathrm{Cov}(X_e, X_{e'}) = \mathrm{Cov}(1_{A_e}\E (X_e ~|~ S_e, S_{e'}), ~&1_{A_{e'}}\E (X_{e'} ~|~ S_e, S_{e'})) ~+ \\
& \E 1_{A_e} 1_{A_{e'}}\mathrm{Cov}(X_e, X_{e'} ~|~ S_e, S_{e'}). 
\end{split}\]

To prove Lemma \ref{l-k-cov-bound}, we prove the following two lemmas.

\begin{lemma}\label{l-k-cov-bound-first-term}
We have
\[\mathrm{Cov}(1_{A_e}\E [X_e ~|~ S_e, S_{e'}], ~1_{A_{e'}}\E [X_{e'} ~|~ S_e, S_{e'}]) \leq C J^{2/3}p_n^2 \frac{(\lambda_n + 1)^k}{n}.  \]
\end{lemma}

\begin{lemma}\label{l-k-cov-bound-second-term}
We have
\[ \E 1_{A_e} 1_{A_{e'}}\mathrm{Cov}(X_e, X_{e'} ~|~ S_e, S_{e'}) \leq C J^{2/3} p_n^2 \rho_k(n).\]
\end{lemma}

\begin{proof}[Proof of Lemma \ref{l-k-cov-bound-first-term}]
The starting point is that $X_e$ is essentially independent of $S_{e'}$, so we should be able to write $\E [X_e ~|~ S_e, S_{e'}] \approx \E [X_e ~|~ S_e]$, and analogously for $X_{e'}$. Towards this end, recall that $X_e$ is a function of 
\[ \bigg(B_k(v, \geen), B_k(v, \geen^e), B_k(v, \geen^F), B_k(v, \geen^{F\cup e})\bigg). \] 
Let us define the approximation $\tilde{X}_e$ to be the same function, applied to 
\[ \bigg(B_k(v, \geen - e'), B_k(v, \geen^e - e'), B_k(v, \geen^F - e'), B_k(v, \geen^{F \cup e} - e')\bigg). \]
Observe that $\tilde{X}_e$ is independent of $S_{e'}$. In the same manner, we way define the approximation $\tilde{X}_{e'}$ which is independent of $S_e$. Now let $Z_e := \E [X_e ~|~ S_e, S_{e'}]$, $\tilde{Z}_e := \E [\tilde{X}_e ~|~ S_e, S_{e'}] = \E [\tilde{X}_e ~|~ S_e]$, and similarly define $Z_{e'}, \tilde{Z}_{e'}$. Let $H_e := H(w_e, w_e')$. Observe that by \ref{l-k-regularity}, we have $|X_e|, |\tilde{X}_e| \leq 1_{A_e} H_e^2$, which implies
\[ Z_e, \tilde{Z}_e \leq 1_{A_e} H_e^2, \]
and similarly for $Z_{e'}, \tilde{Z}_{e'}$. We may write
\[ \begin{split}
\mathrm{Cov}&(1_{A_e} Z_e, 1_{A_{e'}} Z_{e'}) = \mathrm{Cov}(1_{A_e} (Z_e - \tilde{Z}_e), 1_{A_{e'}} \tilde{Z}_{e'}) ~+ \\
&\mathrm{Cov}(1_{A_e} \tilde{Z}_e, 1_{A_{e'}} (Z_{e'} - \tilde{Z}_{e'})) + \mathrm{Cov}(1_{A_e} (Z_e - \tilde{Z}_e), 1_{A_{e'}} (Z_{e'} - \tilde{Z}_{e'})). 
\end{split}\]
To finish, we will bound the three terms on the right hand side. We will only write out how to bound the first term, as the other two terms are bounded similarly. First, observe
\[ |1_{A_e} 1_{A_{e'}} (Z_e - \tilde{Z}_e) \tilde{Z}_{e'}| \leq H_{e'}^2 1_{A_e} 1_{A_{e'}} |Z_e - \tilde{Z}_e|, \]
and
\[ |Z_e - \tilde{Z}_e| = |\E (X_e - \tilde{X}_e ~|~ S_e, S_{e'})| \leq 2H_e^2 \p (X_e \neq \tilde{X}_e ~|~ S_e, S_{e'}). \]
We claim that
\[ \p(X_e \neq \tilde{X}_e ~|~ S_e, S_{e'}) \leq \frac{C}{n} \sum_{j=1}^k \gendeg_n^j \leq \frac{C (\gendeg_n + 1)^k}{n}.\]
Given this claim, putting everything together, we have
\[ | \E1_{A_e} 1_{A_{e'}} (Z_e - \tilde{Z}_e) \tilde{Z}_{e'}| \leq CJ^{2/3} p_n^2 \frac{(\gendeg_n + 1)^{k}}{n}. \]
The term $|\E 1_{A_e} (Z_e - \tilde{Z}_e) \E 1_{A_{e'}} \tilde{Z}_{e'}|$ may be bounded in a similar manner.

To show the claim, observe that the event $\{X_e \neq \tilde{X}_e\}$ implies that $B_k(v, \mbf{G} - e') \neq B_k(v, \mbf{G})$ for at least one of $\mbf{G} = \geen, \geen^e, \geen^F, \geen^{F \cup e}$. Taking say $\mbf{G} = \geen$, we have
\[\p(B_k(v, \geen - e') \neq B_k(v, \geen) ~|~ S_e, S_{e'}) = \p(e' \in B_k(v, \geen) ~|~ b_e, b_{e'}). \]
This may be bounded using arguments similar to those appearing in the proof of Lemma \ref{nbd-perturb-joint-coupling}.
\end{proof}

\begin{proof}[Proof of Lemma \ref{l-k-cov-bound-second-term}]
Define
\[ \mbf{B}_k(-e) := (B_k(v, \geen - e), B_k(v, \geen^F - e), B_k(u, \geen - e), B_k(u, \geen^F - e)), \]
\[ \mbf{B}_k'(-e') := (B_k(v', \geen - e'), B_k(v', \geen^{F'} - e'), B_k(u', \geen - e'), B_k(u', \geen^{F'} - e')). \]
Observe that $X_e$ is a function of $(\mbf{B}_k(-e), S_e)$, and $X_{e'}$ is a function of $(\mbf{B}_k(-e')$, $S_{e'})$. Now define the approximation $\tilde{X}_e$ as the same function applied to $(\mbf{B}_k(-), S_e)$, where
\[ \mbf{B}_k(-) := (B_k(v, \geen - \Delta), B_k(v, \geen^F - \Delta), B_k(u, \geen - \Delta), B_k(u, \geen^F - \Delta)), \]
and $\Delta := \{e, e'\}$. Similarly define the approximation $\tilde{X}_{e'}$ of $X_{e'}$ as a function of $(\mbf{B}'_k(-), S_{e'})$, where
\[ \mbf{B}_k'(-) := (B_k(v', \geen - \Delta), B_k(v', \geen^{F'} - \Delta), B_k(u', \geen - \Delta), B_k(u', \geen^{F'} - \Delta)). \]
As before, we may bound
\[ \p(\mbf{B}_k(-) \neq \mbf{B}_k(-e) ~|~ S_e, S_{e'}) \leq \frac{C(\gendeg_n + 1)^k}{n}, \]
and similarly for $\mbf{B}_k'(-)$. Letting $\tilde{H} := \max(H_e, H_{e'})$, we obtain
\[ \abs{\mathrm{Cov}(X_e, X_{e'} ~|~ S_e, S_{e'}) - \mathrm{Cov}(\tilde{X}_e, \tilde{X}_{e'} ~|~ S_e, S_{e'})} \leq C\tilde{H}^4 \frac{(\gendeg_n + 1)^k}{n}. \]
Thus it suffices to focus on $\tilde{X}_e, \tilde{X}_{e'}$. Observe that by construction, we have that $\mbf{B}_k(-), \mbf{B}_k'(-)$ are independent of $S_e, S_{e'}$. We may thus express
\[ \mathrm{Cov}(\tilde{X}_e, \tilde{X}_{e'} ~|~ S_e, S_{e'}) = \mathrm{Cov}(\Psi(\mbf{B}_k(-)), \Psi'(\mbf{B}_k'(-))), \]
where the functions $\Psi, \Psi'$ depend on $S_e, S_{e'}$, but we hide this dependence. Moreover, we have that $\abs{\Psi}$, $\abs{\Psi'} \leq \tilde{H}^2$, which we now think of as constant when we take the covariance between $\Psi(\mbf{B}_k(-))$ and $\Psi'(\mbf{B}_k'(-))$. Define
\[\mbf{B}_k := (B_k(v, \geen), B_k(v, \geen^F), B_k(u, \geen), B_k(u, \geen^F)), \]
\[ \mbf{B}_k' := (B_k(v', \geen), B_k(v', \geen^{F'}), B_k(u', \geen), B_k(u', \geen^{F'})). \]
We may show
\[ \p(\mbf{B}_k \neq \mbf{B}_k(-)) \leq C \frac{(\gendeg_n + 1)^k}{n}, \]
and similarly for $\mbf{B}_k'$. This allows us to obtain
\[ \abs{\mathrm{Cov}(\Psi(\mbf{B}_k(-)), \Psi'(\mbf{B}_k'(-))) - \mathrm{Cov}(\Psi(\mbf{B}_k), \Psi'(\mbf{B}_k'))} \leq C\tilde{H}^4 \frac{(\gendeg_n + 1)^k}{n}. \]
Now by Lemma \ref{nbd-ind-complicated}, we have
\[ \mathrm{Cov}(\Psi(\mbf{B}_k), \Psi'(\mbf{B}_k')) \leq \tilde{H}^4 \min\bigg(\frac{(\gendeg_n + C)^{2k+C}}{n}, 1\bigg). \] 
The desired result now follows by putting everything together.
\end{proof}

\section{Technical facts about neighborhoods of sparse Erd\H{o}s-R\'{e}nyi graphs}\label{er-graph-facts}

This section collects the key technical facts about neighborhoods of sparse Erd\H{o}s-R\'{e}nyi graphs which are needed. Throughout this section, write $\gendeg_n := np_n$.



The following lemma shows that not only are neighborhoods essentially Galton-Watson trees, but pairs of neighborhoods are essentially independent Galton-Watson trees, in a very quantitative manner. This result seems to be well known and has been proven in \cite{Wast2009} for the case $p_n = 1 - e^{-\gendeg / n}$, but we haven't found a reference which provides a proof for general $np_n \ra \gendeg$. Thus for completeness, we prove it.

\begin{lemma}\label{nbd-joint-coupling}
Suppose $np_n \ra \gendeg$ and $d_{TV}(F_w^{(n)}, F_w) \ra 0$. Fix $k > 0$. For distinct vertices $v, u$, we have a coupling $(B_k(v, \geen), B_k(u, \geen), \btree_v, \btree_u)$ such that $\btree_v, \btree_u \stackrel{i.i.d.}{\sim} \btree(k, \gendeg, F_w)$, and
\[\begin{split} 
\p(B_k(v, \geen) \cong \btree_v, B_k(&u, \geen) \cong \btree_u) \geq 1 - \frac{(2\gendeg + 3)^{k}}{n^{1/3}} ~- \\
& C\frac{(\gendeg_n + 1)^k}{\min(\gendeg, 1)} \bigg(\abs{\gendeg_n - \gendeg} + d_{TV}(F_w^{(n)}, F_w) + \frac{\gendeg^2}{2n}\bigg).
\end{split}\]
\end{lemma}
\begin{proof}
We may assume $\gendeg \leq n$, otherwise the bound is trivial. We first work without edge weights. Let $\tilde{p}_n := 1 - e^{-\gendeg / n}$, and let $\tilde{G}_n$ be the Erd\H{o}s-R\'{e}nyi graph with edge probability $\tilde{p}_n$. For brevity, let $\tilde{B}_k^v := B_k(v, \tilde{G}_n)$, $\tilde{B}_k^u := B_k(u, \tilde{G}_n)$. It follows by Lemma 2.4 of \cite{Wast2009} that we may couple $(\tilde{B}_k^v, \tilde{B}_k^u, T_v, T_u)$ such that $T_v, T_u \stackrel{i.i.d.}{\sim} T(k, \gendeg)$, and
\[ \p(\tilde{B}_k^v \cong T_v, \tilde{B}_k^u \cong T_u) \geq 1 - \frac{(2\gendeg + 3)^k}{n^{1/3}}.\]
Now suppose $p_n \geq \tilde{p}_n$. We may couple $G_n$, $\tilde{G}_n$ in the following manner. If $\tilde{G}_n$ is defined by the edges $\tilde{B} = (\tilde{b}_e, e = (i, j))$, with $\tilde{b}_e \sim \mathrm{Bernoulli}(\tilde{p}_n)$, then define $b_e = \max(\tilde{b}_e, \varep_e)$, with $\varep_e \sim \mathrm{Bernoulli}((p_n - \tilde{p}_n) / (1 - \tilde{p}_n))$. Then $b_e \sim \mathrm{Bernoulli}(p_n)$ as required. Let $B_k^v := B_k(v, G_n)$, $B_k^u := B_k(u, G_n)$. Observe that the event $\{B_k^v \neq \tilde{B}_k^v\}$ is the event that there exists vertices $u_1 \in \tilde{B}_{k-1}^v$, $u_2 \notin \tilde{B}_k^v$, such that $\tilde{b}_{(u_1, u_2)} = 0$, $\varep_{(u_1, u_2)} = 1$. We thus have
\[ \p(B_k^v \neq \tilde{B}_k^v ~|~ \tilde{B}_k^v) \leq n |\tilde{B}_{k-1}^v| \frac{p_n - \tilde{p}_n}{1 - \tilde{p_n}},\]
where $|\tilde{B}_{k-1}^v|$ is the number of vertices in $\tilde{B}_{k-1}^v$. Thus
\[ \p(B_k^v \neq \tilde{B}_k^v) \leq n \frac{p_n - \tilde{p}_n}{1 - \tilde{p_n}} \E |\tilde{B}_{k-1}^v|.\]
By comparison with a branching process, we may bound
\[ \E |\tilde{B}_{k-1}^v| \leq \sum_{j=0}^{k-1} \gendeg_n^j \leq (\gendeg_n + 1)^{k-1}.\]
We have $\tilde{p}_n \geq \frac{\gendeg}{n} - \frac{\gendeg^2}{2n^2}$, and thus (recalling $\lambda \leq n$)
\[n \frac{p_n - \tilde{p}_n}{1 - \tilde{p_n}}  \leq e^{\gendeg / n} \bigg(\gendeg_n - \gendeg + \frac{\gendeg^2}{2n}\bigg) \leq  C \bigg(\abs{\gendeg_n - \gendeg} + \frac{\gendeg^2}{2n}\bigg). \]
We may thus couple $(B_k^v, B_k^u, T_v, T_u)$ such that
\[\p(B_k^v \cong T_v, B_k^u \cong T_u) \geq 1 - \frac{(2\gendeg + 3)^k}{n^{1/3}} - C (\gendeg_n + 1)^{k-1}\bigg(\abs{\gendeg_n - \gendeg} + \frac{\gendeg^2}{2n}\bigg).  \]
Now for each $e$ introduce a coupling $(w_e, \ell_e)$, such that $\p(w_e \neq \ell_e) = d_{TV}(F_w^{(n)}, F_w)$. Let $E_0$ be the event that there is an $e$ in $B_k^v$ or $B_k^u$ such that $w_e \neq \ell_e$. We may naturally couple $(B_k(v, \geen)$, $B_k(u, \geen), \btree_v, \btree_u)$, such that $\btree_v, \btree_u \stackrel{i.i.d.}{\sim} \btree(k, \gendeg, F_w)$, and
\[ \p(B_k(v, \geen) \cong \btree_v, B_k(u, \geen) \cong \btree_u) \geq \p(B_k^v \cong T_v, B_k^u \cong T_u) - \p(E_0).\]
By comparison with a branching process, the expected number of edges in $B_k^v$ is at most $\sum_{j=1}^k \gendeg_n^j \leq (\gendeg_n + 1)^k$. Thus we have
\[ \p(E_0) \leq  2(\gendeg_n + 1)^k d_{TV}(F_w^{(n)}, F_w).\]
To finish, combine the previous results. The case $p_n < \tilde{p}_n$ is handled similarly.
\end{proof}

The following lemma says that if we can couple unconditionally with high probability, then we can couple conditionally with high probability. It was needed in Lemma \ref{remainder-lemma}, in combination with {\localprop}, to show that the remainder terms were small. Recall in (A3) of {\localprop} the definitions of $\btree_k, \tilde{\btree}_k$.

\begin{lemma}\label{nbd-perturb-joint-coupling}
Suppose we have a coupling 
\[ (B_k(v, \geen), B_k(u, \geen), \btree_v, \btree_u), \] 
with $\btree_v, \btree_u \stackrel{i.i.d.}{\sim} \btree(k, \gendeg, F_w)$. Suppose $\varep_k(n)$ is such that
\[ \varep_k(n) \geq 1 - \p(B_k(v, \geen) \cong \btree_v, B_k(u, \geen) \cong \btree_u). \]
Let $e = (v, u)$, and let $e' = (v', u')$ be another edge with vertices distinct from $v, u$. Let $Y_e := (b_e, b_e')$. It is possible to couple $(B_k(v, \geen), B_k(v, \geen^e), \btree, \btree')$ such that 
\[ (\btree, \btree') ~|~ (Y_e = (1, 0), Y_{e'}) \stackrel{d}{=} (\tilde{\btree}_k, \btree_k), \] 
\[ (\btree, \btree') ~|~ (Y_e = (0, 1), Y_{e'}) \stackrel{d}{=} (\btree_k, \tilde{\btree}_k), \]
and furthermore, 
\[ \begin{split}
\p((B_k(v, \geen), B_k(v, \geen^e)) &\cong (\btree, \btree') ~|~ Y_e, Y_{e'}) \geq 1 ~-\\
&\bigg( \jointcouplingprob\bigg). 
\end{split}\]
\end{lemma}
\begin{proof}
To construct $\btree, \btree'$, first take an i.i.d. copy $(W'', B'')$ of $(W, B)$, independent of everything else. Define $\geen''$ to be the weighted graph obtained by using $w_e'', b_e''$ in place of $w_e, b_e$, and $w_{e'}'', b_{e'}''$ in place of $w_{e'}, b_{e'}$. Now obtain a coupling $(B_k(v, \geen''), B_{k-1}(u, \geen''), \btree_k, \btree_{k-1}')$, with $\btree_k, \btree_{k-1}'$ independent, and $\btree_k \stackrel{d}{=} \btree(k, \gendeg, F_w)$, and $\btree_{k-1}' \stackrel{d}{=} \btree(k-1, \gendeg, F_w)$. Observe that $(B_k(v, \geen''), B_{k-1}(u, \geen''))$ is independent of $Y_e, Y_{e'}, w_e, w_e'$, and so we may also assume that $(\btree_k, \btree_{k-1}')$ is independent of $Y_e, Y_{e'}, w_e, w_e'$. Write $B_k := B_k(v, \geen)$, $B_k' := B_k(v, \geen^e)$. There is some function $\Psi$ such that
\[ B_k = \Psi(B_k(v, \geen - e), B_{k-1}(u, \geen - e), w_e, b_e), \]
\[ B_k' = \Psi(B_k(v, \geen - e), B_{k-1}(u, \geen - e), w_e', b_e'). \]
Take a coupling $(w_e, w_e', \ell, \ell')$ independent of everything else such that we have $w_e, w_e' \stackrel{i.i.d.}{\sim} F_w^{(n)}$, $\ell, \ell' \stackrel{i.i.d.}{\sim} F_w$, and
\[ \p(w_e \neq \ell) = \p(w_e' \neq \ell') = d_{TV}(F_w^{(n)}, F_w). \]
Define 
\[ \btree := \Psi(\btree_k, \btree_{k-1}', \ell, b_e), \]
\[ \btree' := \Psi(\btree_k, \btree_{k-1}', \ell', b_e'). \]
Observe that $(\btree, \btree')$ has the desired conditional distribution. Let $E_0$ be the event that $B_k(v, \geen''), B_{k-1}(u, \geen'')$ share a vertex. We have 
\[\begin{split} 
\p((B_k, B_k') \cong (\btree, \btree') ~&|~ Y_e, Y_{e'}) \geq \\
&\p(B_k(v, \geen'') \cong \btree_k, B_{k-1}(u, \geen'') \cong \btree_{k-1}') - \p(E_0) ~- \\
& \p(B_k(v, \geen'') \neq B_k(v, \geen - e) ~|~ Y_e, Y_{e'} ) ~- \\
&\p(B_{k-1}(u, \geen'') \neq B_k(v, \geen - e) ~|~ Y_e, Y_{e'}) ~- \\
&2\p(\ell \neq w_e).  
\end{split}\]
By assumption, we have
\[ \p(B_k(v, \geen'') \cong \btree_k, B_{k-1}(u, \geen'') \cong \btree_{k-1}')  \geq 1 - \varep_k(n).\]
By the union bound, we have
\[ \p(E_0) \leq n\sum_{j_1 = 1}^k \sum_{j_2 = 1}^{k-1} n^{j_1 - 1} n^{j_2 - 1} p_n^{j_1 + j_2} \leq \frac{(\lambda_n + 1)^{2k-1}}{n}. \]
Proceeding, we may bound
\[\begin{split} 
\p(B_k(v, \geen'') &\neq B_k(v, \geen - e) ~|~ Y_e, Y_{e'}) \leq \p(e \in B_k(v, \geen'')) ~+ \\
&\p(e' \in B_k(v, \geen'')) + \p(e' \in B_k(v, \geen - e) ~|~ b_{e'}).
\end{split}\]
We have
\[ \begin{split}
\p(e' \in B_k(v, \geen - e) ~|~ b_{e'}) \leq \p(v' &\in B_k(v, \geen - \{e, e'\})) ~+ \\
&\p(u' \in B_k(v, \geen - \{e, e'\})). 
\end{split} \]
We have
\[ \p(v' \in B_k(v, \geen - \{e, e'\})) \leq \p(v' \in B_k(v, \geen)) \leq \frac{1}{n} \sum_{j=1}^k \gendeg_n^j \leq \frac{(\gendeg_n + 1)^k}{n}. \] 
All other terms may be handled similarly.
\end{proof}



In the rest of the section, we will work towards Lemma \ref{nbd-ind-complicated}, which was needed to bound the covariance between local quantities (Lemma \ref{l-k-cov-bound}, or more specifically, Lemma \ref{l-k-cov-bound-second-term}). The main work is done by Lemma \ref{nbd-joint-ind-complicated}. Instead of proving this straight away, we will first prove the simpler Lemma \ref{coupling-ind-lemma}, where the main idea becomes easier to describe.

With $n$ implicit, we write $B_k^v$ instead of $B_k(v, \geen)$ for brevity. We may explore $B_k^v$ by breadth first search. In other words, from the root $v$, find all neighbors of $v$, and call these the depth 1 vertices. Then find all neighbors of the depth 1 vertices, and call these the depth 2 vertices. Here we specify that if a neighbor of a depth 1 vertex has already been found, then we don't call it a depth 2 vertex. If we can keep exploring in this manner, we obtain an iterative description of $B_k^v$ as follows.

Let $S_k^v$ be the vertex set of $B_k^v$, and let $D_k^v := S_k^v - S_{k-1}^v$ be the set of depth $k$ vertices of $B_k^v$. Given subsets $S_1, S_2$ of the vertex set $V_n$ of $\geen$, let 
\[ X(S_1, S_2) := \{(e, w_e, b_e) : e = (v_1, v_2), v_1 \in S_1, v_2 \in S_2 \}. \]
There is some function $\Psi$ such that for each $k$, we have
\[ B_{k+1}^v = \Psi(B_k^v, X(D_k^v, V_n - S_{k-1}^v)).\]
In words, this says that in the $(k+1)$st iteration, breadth first search explores all edges incident to a vertex in $D_k^v$, i.e. a depth $k$ vertex. Moreover, we only need to look at edges which connect $D_k^v$ and $V_n - S_{k-1}^v$. This is because edges between $D_k^v$ and $S_{k-1}^v$ have already been explored by previous iterations. We may use this iterative description of $B_k^v$ to obtain the following lemma.

\begin{lemma}\label{coupling-ind-lemma}
For each $k$, there is a coupling of $(B_{k+1}^v, B_{k+1}^u, \tilde{B}_{k+1}^v, \tilde{B}_{k+1}^u)$ that satisfies the following properties. Let $I_k := \{S_k^v \cap S_k^u = \varnothing\}$. Then on $I_k$, $\tilde{B}_{k+1}^v, \tilde{B}_{k+1}^u$ are conditionally independent given $B_k^v, B_k^u$. Moreover, on the event $I_k$, the conditional law of $\tilde{B}_{k+1}^v$ given $B_k^v, B_k^u$ is the conditional law of $B_{k+1}^v$ given $B_k^v$, and the conditional law of $\tilde{B}_{k+1}^u$ given $B_k^v, B_k^u$ is the conditional law of $B_{k+1}^u$ given $B_k^u$. In other words, for bounded measurable functions $g_v, g_u$, we have
\[ 1_{I_k} \mathrm{Cov}(g_v(\tilde{B}_{k+1}^v), g_u(\tilde{B}_{k+1}^u) ~|~ B_k^v, B_k^u) = 0, \] 
\[ 1_{I_k} \E(g_v(\tilde{B}_{k+1}^v) ~|~ B_k^v, B_k^u) = 1_{I_k} \E (g_v(B_{k+1}^v) ~|~ B_k^v),  \]
\[ 1_{I_k} \E(g_u(\tilde{B}_{k+1}^u) ~|~ B_k^v, B_k^u) = 1_{I_k} \E (g_u(B_{k+1}^u) ~|~ B_k^u). \]
Finally, we have
\[ 1_{I_k} \p(\tilde{B}_{k+1}^v \neq B_{k+1}^v ~|~ B_k^v, B_k^u) \leq 1_{I_k} C|S_k^v| |S_k^u| p_n, \]
\[ 1_{I_k} \p(\tilde{B}_{k+1}^u \neq B_{k+1}^u ~|~ B_k^v, B_k^u) \leq 1_{I_k} C|S_k^v| |S_k^u| p_n. \]
\end{lemma}
\begin{remark}
The main idea of the proof is noting that as long as $B_k^v, B_k^u$ don't intersect, the objects in next iteration $B_{k+1}^v, B_{k+1}^u$ are very weakly interacting with each other. Moreover, the amount of interaction is governed by the size of $B_k^v, B_k^u$. We can remove these interactions by re-randomization, and if the sizes of $B_k^v, B_k^u$ are not too large, then this re-randomization is unlikely to cause changes.
\end{remark}
\begin{proof}
We first show how to generate the pair $(B_{k+1}^v, B_{k+1}^u)$ starting from $B_k^v, B_k^u$, on the event $I_k$. Let
\[ X_1 := X(D_k^v, V_n - S_{k-1}^v - S_k^u), \]
\[ X_2 := X(D_k^u, V_n - S_{k-1}^u - S_k^v), \]
\[ X_3 := X(D_k^v, D_k^u). \]
Then on $I_k$, we have
\[ B_{k+1}^v = \Psi(B_k^v, X_1 \cup X_3),\]
\[ B_{k+1}^u = \Psi(B_k^u, X_2 \cup X_3). \]
Note $X_1 \cup X_3 = X(D_k^v, V_n - S_{k-1}^v - S_{k-1}^u)$, and $X_2 \cup X_3 = X(D_k^u, V_n - S_{k-1}^u - S_{k-1}^v)$. The point is that on $I_k$, there are some further restrictions on which edges can be present. In other words, there can be no edges between $D_k^v$ and $S_{k-1}^u$, and there can be no edges between $D_k^u$ and $S_{k-1}^v$. Note that on $I_k$, the objects $X_1, X_2$ are conditionally independent given $B_k^v, B_k^u$.

We now construct $(\tilde{B}_{k+1}^v, \tilde{B}_{k+1}^u)$. First, let $\tilde{X}_3$ be an i.i.d. copy of $X_3$, conditional on $B_k^v, B_k^u$. Observe then that $\Psi(B_k^v, X_1 \cup X_3)$, $\Psi(B_k^u, X_2 \cup \tilde{X}_3)$ are conditionally independent (at least on $I_k$). However, the conditional law of $\Psi(B_k^v, X_1 \cup X_3)$ is not as desired. To correct this, we re-randomize the edges between $D_k^v$ and $S_{k-1}^u$. In other words, take $(W', B')$ an i.i.d. copy of $(W, B)$. Define
\[ X_1' := \{(e, w_e', b_e') : e = (v_1, v_2), v_1 \in D_k^v , v_2 \in S_{k-1}^u\}. \]
Similarly, let
\[ X_2' := \{(e, w_e', b_e') : e = (v_1, v_2), v_1 \in D_k^u, v_2 \in S_{k-1}^v \}. \]
We now set
\[ \tilde{B}_{k+1}^v := \Psi(B_k^v, X_1 \cup X_1' \cup X_3), \]
\[ \tilde{B}_{k+1}^u := \Psi(B_k^u, X_2 \cup X_2' \cup \tilde{X}_3). \]
The point is that to obtain $\tilde{B}_{k+1}^v$ from $B_k^v$, we no longer include the restrictions on the edges between $D_k^v$ and $S_{k-1}^u$ that are induced by the event $I_k$, and thus the law of $\tilde{B}_{k+1}^v$ given $B_k^v, B_k^u$ is exactly the law of $B_{k+1}^v$ given $B_k^v$. The analogous is true for $\tilde{B}_{k+1}^u$.

To finish, we need to show that $\tilde{B}_{k+1}^v \neq B_{k+1}^v$ with very low probability. This event only happens if one of the re-randomized edges between $D_k^v$ and $S_{k-1}^u$ is present, i.e. there must exist some $e = (v_1, v_2)$, with $v_1 \in D_k^v$, $v_2 \in S_{k-1}^u$, such that $b_e' = 1$. A union bound now does the trick. A similar argument works for $\tilde{B}_{k+1}^u$.
\end{proof}

By applying Lemma \ref{coupling-ind-lemma}, we can obtain the following lemma.

\begin{lemma}\label{covariance-iteration}
Let $g_v, g_u$ be measurable functions which are bounded in absolute value by 1. Then
\[ \begin{split}
\mathrm{Cov}(g_v(B_{k+1}^v), g_u(B_{k+1}^u)) \leq C &(\p(I_k^c) + p_n \E (|S_k^v| |S_k^u|)) ~+ \\
&\mathrm{Cov}(\E (g_v(B_{k+1}^v) ~|~ B_k^v), \E (g_u(B_{k+1}^u) ~|~ B_k^u)).
\end{split}\]
\end{lemma}

The following lemma says that the event $I_k$ happens with very high probability, and that the sizes of $|S_k^v|, |S_k^u|$ are not too large.

\begin{lemma}\label{galton-watson-size-lemma}
For any $k > 0$, we have
\[ \p(I_k^c) \leq \frac{1}{n} \sum_{l=1}^{2k} \lambda_n^l \leq \frac{(\gendeg_n + 1)^{2k}}{n}, \]
\[ \E (|S_k^v| |S_k^u|) \leq Ck (\gendeg_n + C)^{2k+C} \leq  (\gendeg_n + C)^{2k+C}. \]
\end{lemma}
\begin{proof}
The first assertion follows by a union bound over all possible paths from $v$ to $u$ that use at most $2k$ edges. The second assertion follows first by Cauchy-Schwarz, and then noting that $|S_k^v|, |S_k^u|$ are stochastically dominated by the total number of vertices in a depth $k$ Galton-Watson tree, with offspring distribution $\mathrm{Binomial}(n-1, p_n)$, and then concluding by using standard formulas for Galton-Watson trees.
\end{proof}


By iterating Lemma \ref{covariance-iteration} and applying Lemma \ref{galton-watson-size-lemma}, we obtain the following.

\begin{lemma}\label{covariance-sup-bound}
For any $k > 0$, we have
\[ \sup_{g_v, g_u} \mathrm{Cov}(g_v(B_k^v), g_u(B_k^u)) \leq  \frac{C k (\lambda_n + C)^{2k+C}}{n} \leq \frac{ (\gendeg_n + C)^{2k+C}}{n}.\]
Here the supremum is taken over all pairs of measurable functions with absolute value bounded by 1.
\end{lemma}

With all the notation set, we now go on a slight diversion and quickly prove the following lemma. It is needed for Lemma \ref{remainder-lemma}.

\begin{lemma}\label{nbd-tree-prob}
We have
\[ \p(B_k^v \text{ is not a tree}) \leq \frac{(\gendeg_n + C)^{2k+C}}{n}. \]
\end{lemma}
\begin{proof}
Let $A_k$ be the event that $B_k^v$ is a tree. We have
\[ 1_{A_{k-1}} \p(A_k^c ~|~ B_{k-1}^v) \leq 1_{A_{k-1}} \binom{|D_{k-1}^v|}{2} p_n \leq 1_{A_{k-1}} |D_{k-1}^v|^2 p_n. \]
Upon taking expectations and iterating, we obtain (note $\p(A_1^c) = 0$)
\[ \p(A_k^c) \leq  p_n \E \sum_{j=1}^{k-1} |D_j^v|^2 . \]
Observe
\[ \sum_{j=1}^{k-1} |D_j^v|^2 \leq |S_{k-1}^v|^2, \]
and finish by observing that $\E |S_{k-1}^v|^2 \leq (\gendeg_n + C)^{2k+C}$, as noted in the proof of Lemma \ref{galton-watson-size-lemma}.
\end{proof}



By rewriting the proof of Lemma \ref{coupling-ind-lemma} in a more general form, we may obtain the following.

\begin{lemma}\label{nbd-joint-ind-complicated}
Let $E_n := \{(i, j), 1 \leq i < j \leq n\}$. Let $e_0 = (v_0, u_0)$, $e'_0 = (v'_0, u'_0)$ have distinct vertices. Let $F \sse E_n - \{e_0\}$, $F' \sse E_n - \{e'_0\}$. Define
\[ \mbf{B}_k := (B_k(v_0, \geen), B_k(v_0, \geen^F), B_k(u_0, \geen), B_k(u_0, \geen^F)), \]
\[ \mbf{B}_k' := (B_k(v'_0, \geen), B_k(v_0', \geen^{F'}), B_k(u'_0, \geen), B_k(u'_0, \geen^{F'})). \]
Let $N_k$ be the number of vertices in $\mbf{B}_k$, and $N_k'$ the number of vertices in $\mbf{B}_k'$. Let $I_k$ be the event that the vertex sets of $\mbf{B}_k$ and $\mbf{B}_k'$ intersect. There is a coupling $(\mbf{B}_{k+1}, \mbf{B}_{k+1}', \tilde{\mbf{B}}_{k+1}, \tilde{\mbf{B}}_{k+1}')$ such that on $I_k$, we have that $\tilde{\mbf{B}}_{k+1}, \tilde{\mbf{B}}_{k+1}'$ are conditionally independent given $\mbf{B}_k, \mbf{B}_k'$, and the law of $\tilde{\mbf{B}}_{k+1}$ given $\mbf{B}_{k}$, $\mbf{B}_{k}'$ is the law of $\mbf{B}_{k+1}$ given $\mbf{B}_k$, and the law of $\tilde{\mbf{B}}_{k+1}'$ given $\mbf{B}_{k}, \mbf{B}_{k}'$ is the law of $\mbf{B}_{k+1}'$ given $\mbf{B}_k'$. Moreover, we have
\[ 1_{I_k} \p(\tilde{\mbf{B}}_{k+1} \neq \mbf{B}_{k+1} ~|~ \mbf{B}_{k}, \mbf{B}_{k}') \leq 1_{I_k} C N_k N_k' p_n,  \]
\[ 1_{I_k} \p(\tilde{\mbf{B}}_{k+1}' \neq \mbf{B}_{k+1}' ~|~ \mbf{B}_{k}, \mbf{B}_{k}') \leq 1_{I_k} C N_k N_k' p_n.  \]
\end{lemma}
\begin{proof}
Let $V_n$ be the vertex set of $\geen$.
Let $S_k$ be the vertex set of $\mbf{B}_k$ (more precisely, the union of the vertex sets of the four graphs which make up $\mbf{B}_k$), and let $S_k'$ be the vertex set of $\mbf{B}_k'$.
For $S_1, S_2 \sse V_n$, define
\[ X(S_1, S_2) := \{(e, w_e, b_e, w_e', b_e') : e = (v, u), v \in S_1, u \in S_2  \}. \]
There is some function $\Psi$, which depends on $F$, such that
\[ \mbf{B}_{k+1} = \Psi(\mbf{B}_k, X(S_k, V_n)).  \]
Similarly, there is a function $\Psi'$ which depends on $F'$ such that
\[ \mbf{B}_{k+1}' = \Psi(\mbf{B}_k', X(S_k', V_n)). \]
Define
\begin{align*}
X_1 &:= X(S_k, V_n - S_k'),  \\
X_2 &:= X(S_k', V_n - S_k),  \\
X_3 &:= X(S_k, S_k'). 
\end{align*}
Observe then that
\[ X(S_k, V_n) = X_1 \cup X_3, ~~ X(S_k', V_n)  = X_2 \cup X_3. \]
Moreover, note that on $I_k$, we have that $X_1, X_3$ are conditionally independent given $\mbf{B}_k, \mbf{B}_k'$. We may thus construct $\tilde{\mbf{B}}_{k+1}, \tilde{\mbf{B}}_{k+1}'$ as follows. Conditional on $\mbf{B}_k, \mbf{B}_k'$,  let $\tilde{X}_3$ be distributed as $X_3$, conditional on $\mbf{B}_k$, and let $\tilde{X}_3'$ be distributed as $X_3$, conditional on $\mbf{B}_k'$. Moreover, let $\tilde{X}_3, \tilde{X}_3'$ be independent of each other and everything else, conditional on $\mbf{B}_k, \mbf{B}_k'$. Then on $I_k$, define
\[ \tilde{\mbf{B}}_{k+1} := \Psi(\mbf{B}_k, X_1 \cup \tilde{X}_3), ~ \tilde{\mbf{B}}_{k+1}' := \Psi'(\mbf{B}_k', X_2 \cup \tilde{X}_3').   \]
By construction, on the event $I_k$, we have that $\tilde{\mbf{B}}_{k+1}, \tilde{\mbf{B}}_{k+1}'$ are conditionally independent given $\mbf{B}_k, \mbf{B}_k'$. Moreover, observe that on the event $I_k$, conditional on $\mbf{B}_k, \mbf{B}_k'$, we have that $X_1 \cup \tilde{X}_3$ has the law of $X(S_k, V_n)$ conditional only on $\mbf{B}_k$. Thus on $I_k$, the law of $\tilde{\mbf{B}}_{k+1}$ conditional on $\mbf{B}_k, \mbf{B}_k'$ is exactly the law of $\mbf{B}_{k+1}$ conditional on $\mbf{B}_k$. The analogous statement is true for $\tilde{\mbf{B}}_{k+1}'$. 

To finish, we need to show
\[ 1_{I_k} \p(\tilde{\mbf{B}}_{k+1} \neq \mbf{B}_{k+1} ~|~ \mbf{B}_{k}, \mbf{B}_{k}') \leq 1_{I_k} C N_k N_k' p_n. \]
The proof for $\tilde{\mbf{B}}_{k+1}'$ will be the exact same. To set notation, write
\[ \tilde{X}_3 = \{(e, \tilde{w}_e, \tilde{b}_e, \tilde{w}_e', \tilde{b}_e') : e = (v, u), v \in S_k, u \in S_k' \}. \]
Observe that if for all $e = (v, u)$, $ v\in S_k$, $u \in S_k'$, we have $b_e, b_e', \tilde{b}_e, \tilde{b}_e' = 0$, then necessarily $\tilde{\mbf{B}}_{k+1} = \mbf{B}_{k+1}$. Thus it suffices to bound the probability that this event doesn't happen. 
The point is that on the event $I_k$, for $e = (v, u)$, $v \in S_k$, $u \in S_k'$, the conditional distribution of any of the $b_e, b_e', \tilde{b}_e, \tilde{b}_e'$ given $\mbf{B}_k, \mbf{B}_k'$ is either $\mathrm{Bernoulli}(p_n)$, or the point mass at 0 (in words, either the edge $e$ is left unrestricted, or it forced to not be present). We now finish by the union bound, along with the fact that if $U$ is a random variable whose distribution is either identically 0 or $\mathrm{Bernoulli}(p_n)$, then 
\[ \p(U = 1) \leq p_n. \qedhere\]
\end{proof}

We may now deduce the following lemma from Lemma \ref{nbd-joint-ind-complicated} in the same way we deduced Lemma \ref{covariance-sup-bound} from Lemma \ref{coupling-ind-lemma}. Here we additionally use the fact that if two random variables lie in the interval $[-1, 1]$, then their covariance must also be in $[-1, 1]$.

\begin{lemma}\label{nbd-ind-complicated}
For any $k > 0$, we have
\[ \sup_{g, g'} \mathrm{Cov}(g(\mbf{B}_{k}), g'(\mbf{B}_k')) \leq  \min\bigg(\frac{(\gendeg_n + C)^{2k+C}}{n}, 1 \bigg). \]
Here the supremum is taken over all pairs of measurable functions which have absolute value bounded by 1.
\end{lemma}

\section{Concluding remarks}\label{section:concluding-remarks}

A natural direction for future work is to try to prove a central limit theorem for minimum matching in the mean field setting, following the same strategy as was used for optimal edge cover. One of the main difficulties is in proving the analog of Proposition \ref{v-lambda-exp-conv} (recall this proposition allowed us to take $\lambda \toinf$ with $n$) for minimum matching. To do so, we need to analyze the following operator (see \cite{ParWast2017, Wast2009, Wast2012}). Let $\lambda > 0$. Given $F : [-\lambda / 2, \lambda / 2] \ra [0, 1]$, define $V_\lambda F : [-\lambda / 2, \lambda / 2] \ra [0, 1]$ by the following:
\[ (V_\lambda F)(x) := \exp(-\int_{-x}^{\lambda / 2} F(\ell) d\ell). \]
From some simulations, we don't think the analog of Proposition \ref{v-lambda-exp-conv} is actually true for this operator $V_\lambda$, because it seems that $\alpha(\lambda)$ in fact converges to 1 as $\lambda \toinf$. The difficulty is then trying to understand the rate of convergence of $\alpha(\lambda)$, i.e. does it behave like $1 - \frac{1}{\lambda}$, or $1 - \frac{1}{\log \lambda}$, or something else.

Another direction is to consider vertex-weighted graphs instead of edge-weighted graphs. For example, \cite{GNS2005} proves the long-range independence property for the maximum weight independent set problem, when the average vertex degree (i.e. $\gendeg$) is at most $2e$. This problem is a combinatorial optimization problem on vertex-weighted graphs. All the arguments in proving Theorem \ref{main-result} should carry over with small modifications to the vertex-weighted case; we decided not to include this in the paper because we couldn't figure out a good way to have one reasonable set of notation that covers both cases.

It is also possible to apply Theorem \ref{main-result} to functions of sparse random graphs which are not combinatorial optimization problems. For example, Dembo and Montanari \cite{DM2010} use a form of the Objective method to compute limiting constants for the free energy of Ising models on locally tree-like graphs (this includes sparse Erd\H{o}s-R\'{e}nyi graphs). One may use the results  of \cite{DM2010} to verify {\localprop} for the free energy of the Ising model on a sparse Erd\H{o}s-R\'{e}nyi graph. In particular, Theorem 3.1 of \cite{DM2010} can be used to check (A1), and Lemma 4.3 of \cite{DM2010} can be used to check (A3). Thus one may prove a central limit theorem for the free energy.

Finally, a natural follow up question would be to try to determine the asymptotic behavior of the variance. Note in our setting $n$ will be the right order for the variance, since we are assuming that the variance is at least order n, and by the assumption \eqref{f-regularity} and the Efron-Stein inequality, the variance will be at most order n. We would thus expect that $\Var{f(\geen)} / n \ra \sigma^2$, where $\sigma^2$ is a constant. This would then allow us to divide by $n^{1/2}$ rather than $(\Var{f(\geen)})^{1/2}$ in our central limit theorems. Moreover, we might expect that the constant $\sigma^2$ can be explicitly determined. The natural approach to showing such a result would be to try to use the Objective method, since as mentioned in the introduction it has been very successful for analyzing means. However, Aldous and Steele write in their survey on the Objective method \cite[Section 7]{AldSte2004} that it cannot be used to determine the asymptotic behavior of the variance. Thus it seems that there currently does not exist a general method for analyzing the variance. Moreover, even problem specific results are few and far between. In the case of minimum matching, the only results we are aware of are \cite{Wast2005b, Wast2010}, where it is further assumed that the edge weights are Exponential (though to be fair, exact finite $n$ formulas are proven under this assumption), to take advantage of special properties of the Exponential distribution. However, it is expected, but not rigorously proven, that the formula for $\sigma^2$ should still hold for more general edge weight distributions (see \cite{MPS2019}). Thus there is still much to be understood about the variance of random optimization problems.




\section*{Acknowledgments}

We thank Sourav Chatterjee for helpful conversations and encouragement. We also thank the anonymous referees for many helpful comments and suggestions.

\addtocontents{toc}{\protect\setcounter{tocdepth}{0}}


\begin{thebibliography}{37}

    \bibitem{AddBerr2013}
        {Addario-Berry, L.} (2013).
        {The local weak limit of the minimum spanning tree of the complete graph.}
        Preprint available at \href{https://arxiv.org/abs/1301.1667}{arXiv:1301.1667}.

    \bibitem{Ald2001} 
        {Aldous, D. J.} (2001).
        {The $\zeta(2)$ limit in the random assignment problem}
        {\em Random Structures Algorithms}, {\bf 18} no. 4, 381-418.

	\bibitem{AldBan2005}
		{Aldous, D. J., Bandyopadhyay, A.} (2005).
		{A survey of max-type recursive distributional equations.}
		{\em Ann. Appl. Probab.}, {\bf 15} no. 2, 1047 - 1110.
		
	\bibitem{AldSte2004}
	    {Aldous, D. J., Steele, J.M.} (2004).
	    {The objective method: probabilistic combinatorial optimization and local weak convergence.}
	    {\em Probability on Discrete Structures, Encyclopaedia Math. Sci.} {\bf 110} {Springer-Verlag},
	    {New York}, 1-72.
	    
	\bibitem{Alex1996}
	    {Alexander, K.S.} (1996).
	    {The RSW theorem for continuum percolation and the CLT for Euclidean minimal spanning trees.}
	    {\em Ann. Appl. Probab.}, {\bf 6} no. 2, 466-494.
	    
	\bibitem{AST2019}
	    {Ambrosio, L., Stra, F., Trevisan, D.} (2019).
	    {A PDE approach to a 2-dimensional matching problem.}
	    {\em Probab. Theory Related Fields}, {\bf 173} no. 1-2, 433-477.
	    
    \bibitem{AY2018}
        {Athreya, S., Yogeshwaran, D.} (2018).
        {Central limit theorem for statistics of subcritical configuration models.}
        Preprint available at \href{https://arxiv.org/abs/1808.06778}{arXiv:1808.06778}.
        

    \bibitem{BR2019}
        {Barbour, A.D., R\"{o}llin, A.} (2019).
        {Central limit theorems in the configuration model.}
        {\em Ann. Appl. Probab.}, {\bf 29} no. 2, 1046-1069. 

	\bibitem{FODS} 
		{Blum, A., Hopcroft, J., Kannan, R.} (2018).
		{Foundations of Data Science.}
		{\em Cambridge University Press.}
		
	\bibitem{CLPS2014}
	    {Caracciolo, S., Lucibello, C., Parisi, G., Sicuro, G.} (2014).
	    {Scaling hypothesis for the Euclidean bipartite matching problem.}
	    {\em Physical Review E}, {\bf 90} 012118.
	    
		

	\bibitem{Ch2008}
		{Chatterjee, S.} (2008).
		{A new method of normal approximation.}
		{\em Ann. Probab.}, {\bf 36} no. 4, 1584-1610.

	\bibitem{Ch2014}
		{Chatterjee, S.} (2014).
		{A short survey of Stein's method.}
		{\em Proceedings of ICM 2014}, Vol IV, 1-24. 

	\bibitem{Ch2018}
		{Chatterjee, S.} (2019).
		{A general method for lower bounds on fluctuations of random variables.}
		{\em Ann. Probab.}, {\bf 47} no. 4, 2140-2171.
		
	\bibitem{ChSen2013}
	    {Chatterjee, S., Sen, S.} (2017).
	    {Minimal spanning trees and Stein's method.}
	    {\em Ann. Appl. Probab.}, {\bf 27} no. 3, 1588-1645.
	    
	\bibitem{DBL2019}
        {Del Barrio, E., Loubes, J.-M.} (2019).
        {Central limit theorems for empirical transportation cost in general dimension.}
        {\em Ann. Probab.}, {\bf 47} no. 2, 926-951.
    
    \bibitem{DM2010}
        {Dembo, A., Montanari, A.} (2010).
        {Ising models on locally tree-like graphs.}
        {\em Ann. Appl. Probab.}, {\bf 20} no. 2, 565-592.
        
    \bibitem{Frieze1985}
        {Frieze, A.M.} (1985).
        {On the value of a random minimum spanning tree problem.}
        {\em Discrete Appl. Math.}, {\bf 10} no. 1, 47-56.
		
	\bibitem{GNS2005}
		{Gamarnik, D., Nowicki, T., Swirszcz, G.} (2005).
		{Maximum weight independent sets and matchings in sparse random graphs. Exact results using the local weak convergence method.}
		{\em Random Structures Algorithms}, {\bf 28} no. 1, 76-106.
		
	\bibitem{HessWast2008}
		{Hessler, M., W\"{a}stlund, J.} (2008).
		{Concentration of the cost of a random matching problem.}
		Preprint available at \url{http://www.math.chalmers.se/~wastlund/martingale.pdf}.
		
	\bibitem{HessWast2010}
		{Hessler, M., W\"{a}stlund, J.} (2010).
		{Edge cover and polymatroid flow problems.}
		{\em Electron. J. Probab.}, {\bf 15} no. 72, 2200-2219.
		
	\bibitem{Janson1995}
	    {Janson, S.} (1995).
	    {The minimal spanning tree in a complete graph and a functional limit theorem for trees in a random graph.}
	    {\em Random Structures Algorithms}, {\bf 7} no. 4, 337-355.
		
	\bibitem{JasTat2011}
	    {Jaslar, S., Tatikonda, S.} (2011).
	    {Maximum weight partial colorings on sparse random graphs.}
	    {\em SIAM J. Discrete Math.}, {\bf 25} no. 2, 934-955.
	    
    \bibitem{KestLee1996}
        {Kesten, H., Lee, S.} (1996).
        {The central limit theorem for weighted minimal spanning trees on random points.}
        {\em Ann. Appl. Probab.}, {\bf 6} no. 2, 495-527.
	
	\bibitem{Khand2014} 
	    {Khandwawala, M.} (2014).
	    {Belief propagation for minimum weight many-to-one matchings in the random complete graph.}
	    {\em Electron. J. Probab.}, {\bf 19} no. 112, 1-40.
	
	\bibitem{KhanSund2014}
		{Khandwawala, M., Sundaresan, R.} (2014).
		{Belief propagation for the optimal edge cover in the random complete graph.}
		{\em Ann. Probab.}, {\bf 24} no. 6, 2414-2454.
		
    \bibitem{LSY2019}
        {L\`{a}chieze-Rey, R., Schulte, M., Yukich, J.E.} (2019)
        {Normal approximation for stabilizing functionals.}
        {\em Ann. Appl. Probab.}, {\bf 29} no. 2, 931-993.
		
	\bibitem{LinWast2003}
	    {Linusson, S., W\"{a}stlund, J.} (2003).
	    {A proof of Parisi's conjecture on the random assignment problem.}
	    {\em Probab. Theory Related Fields}, {\bf 128} no. 3, 419-440.
	    
	    
	\bibitem{MPS2019}
	    {Malatesta, E.M., Parisi, G., Sicuro, G.} (2019).
	    {Fluctuations in the random-link matching problem.}
	    {\em Phys. Rev. E}, {\bf 100} 032102.
	    
    \bibitem{NPS2005}
	    {Nair, C., Prabhakar B., Sharma, M.} (2005).
	    {Proofs of the Parisi and Coppersmith-Sorkin random assignment conjectures.}
	    {\em Random Structures Algorithms}, {\bf 27} 413-444.
		
	\bibitem{ParWast2017}
		{Parisi, G., W\"{a}stlund, J.} (2017)
		{Mean field matching and TSP in pseudo-dimension 1.}
		Preprint available at
		\href{https://arxiv.org/abs/1801.00034}{arXiv:1801.00034}.
		
		

    \bibitem{Salez2013}
        {Salez, J.} (2013).
        {Weighted enumeration of spanning subgraphs in locally tree-like graphs.}
        {\em Random Structures Algorithms}, {\bf 43} no. 3, 377-397.
	    
    \bibitem{Steele1997}
        {Steele, J.M.} (1997).
        {Probability theory and combinatorial optimization.}
        {\em NSF-CBMS Volume 69.} {Society for Industrial and Applied Mathematics (SIAM)}, Philadelphia, PA.
		
	\bibitem{Wast2005}
	    {W\"{a}stlund, J.} (2005).
	    {A proof of a conjecture of Buck, Chan, and Robbins on the expected value of the minimum assignment.}
	    {\em Random Structures Algorithms}, {\bf 26} no. 1-2, 237-251.
	    
	\bibitem{Wast2005b}
	    {W\"{a}stlund, J.} (2005).
	    {The variance and higher moments in the random assignment problem.}
	    {\em Link\"{o}ping Studies in Mathematics}, 8. 
	    
	\bibitem{Wast2008}
	    {W\"{a}stlund, J.} (2008).
	    {Random matching problems on the complete graph.}
	    {\em Electron. Commun. Probab.}, {\bf 13} 258-265.
	
	\bibitem{Wast2009}
		{W\"{a}stlund, J.} (2009).
		{Replica symmetry and combinatorial optimization.}
		Preprint available at \href{https://arxiv.org/abs/0908.1920}{arXiv:0908.1920}.
		
	\bibitem{Wast2009b}
	    {W\"{a}stlund, J.} (2009).
	    {An easy proof of the $\zeta(2)$ limit in the random assignment problem.}
	    {\em Electron. Commun. Probab.}, {\bf 14} 261-269. 
	
	\bibitem{Wast2010}
	    {W\"{a}stlund, J.} (2010).
	    {The mean field traveling salesman and related problems.}
	    {\em Acta Math.}, {\bf 204} no. 1, 91-150.
		
	\bibitem{Wast2012}
		{W\"{a}stlund, J.} (2012).
		{Replica symmetry of the minimum matching.}
		{\em Ann. of Math. (2)}, {\bf 175} no. 3, 1061-1091.
		
	\bibitem{Yuk1998}
	    {Yukich, J.E.} (1998).
	    {Probability theory of classical euclidean optimization problems.}
	    {\em Lecture Notes in Mathematics}, {\bf 1675} Springer-Verlag, New York.
	    
	\bibitem{Yuk2013}
	    {Yukich, J.E.} (2013).
	    {Limit theorems in discrete stochastic geometry.}
	    {In {\em Stochastic Geometry, Spatial Statistics, and Random Fields. Lectures Notes in Math.}} {\bf 2068} 239-275. Heidelberg: Springer.
		

\end{thebibliography}
\end{document}